\documentclass[reqno]{amsart}

\usepackage{a4wide}
\usepackage{graphicx}
\usepackage{subfigure}
\usepackage{listings}
\usepackage{verbatim}

\title[Variational scheme for drift diffusion]{Convergence of a variational Lagrangian scheme \\ for a nonlinear drift diffusion equation}

\author{Daniel Matthes}
\address{Daniel Matthes \\ Zentrum Mathematik \\ TU M\"unchen \\ Boltzmannstr. 3 \\ D-85748 Garching \\ Germany}
\email{matthes@ma.tum.de}

\author{Horst Osberger}
\address{Horst Osberger \\ Zentrum Mathematik \\ TU M\"unchen \\ Boltzmannstr. 3 \\ D-85748 Garching \\ Germany}
\email{osberger@ma.tum.de}

\date{\today}

\begin{document}

\newcommand{\setR}{\mathbb{R}}
\newcommand{\setN}{\mathbb{N}}
\newcommand{\dd}{\,\mathrm{d}}
\newcommand{\indy}{\mathbf{1}}

\newcommand{\theX}{\mathrm{X}}
\newcommand{\xvec}{\vec{\mathrm x}}
\newcommand{\yvec}{\vec{\mathrm y}}
\newcommand{\gvec}{\vec{\mathrm z}}

\newcommand{\convf}{\mathbf{u}}
\newcommand{\convg}{\mathbf{z}}
\newcommand{\convX}{\mathbf{X}}

\newcommand{\theh}{{\boldsymbol \xi}}
\newcommand{\xspc}{\mathfrak{X}}
\newcommand{\xspcN}{\mathfrak{X}_\theh}
\newcommand{\xseq}{\mathfrak{x}}
\newcommand{\dens}{\mathrm{D}^M(I)}
\newcommand{\densN}{\mathrm{D}^M_\theh(I)}
\newcommand{\densNj}{\mathrm{D}^M_{\theh_j}(I)}
\newcommand{\densNintro}{\mathrm{D}^M_{\theh^{(j)}}(I)}
\newcommand{\prb}{\mathcal{P}(I)}
\newcommand{\tst}{\mathcal{D}}

\newcommand{\tv}[1]{\{#1\}_{\mathrm{TV}}}
\newcommand{\tvn}[1]{\|#1\|_{\mathrm{TV}}}

\newcommand{\Wmat}{\mathrm{W}}
\newcommand{\Wnew}{\widetilde{\mathrm W}}
\newcommand{\Imat}{\mathrm{I}}
\newcommand{\Dmat}{\mathrm{D}}
\newcommand{\Qmat}{\mathrm{Q}}

\newcommand{\anrj}{\mathbf{E}}
\newcommand{\nrj}{\mathbb{E}}
\newcommand{\auxil}{\mathbf{A}}

\newcommand{\prss}{\mathrm{P}}
\newcommand{\nrjbelow}{\underline{\mathbf{E}}}
\newcommand{\aF}{\widetilde{\mathbf{F}}}
\newcommand{\wass}{\mathbf{W}_2}
\newcommand{\grd}{\partial_{\xvec}}
\newcommand{\hess}{\partial^2_{\xvec}}
\newcommand{\nci}{\mathfrak{F}}
\newcommand{\delmin}{{\underline\delta}}
\newcommand{\delmax}{{\overline\delta}}
\newcommand{\supp}{\operatorname{supp}}
\newcommand{\velo}{\vec{v}}
\newcommand{\flow}{\mathfrak{S}}
\newcommand{\testf}{\varphi}
\newcommand{\err}{\mathsf{err}}
\newcommand{\bigo}{\mathcal{O}}
\newcommand{\var}[1]{\operatorname{Var}_a^b\big(#1\big)}
\newcommand{\thec}{\mathcal{C}}
\newcommand{\apt}{p}
\newcommand{\eps}{\varepsilon}

\newcommand{\hatf}{\theta}
\newcommand{\hatF}{\Theta}

\newcommand{\essinf}{\operatorname*{ess\,inf}}

\newtheorem{thm}{Theorem}
\newtheorem{prp}[thm]{Proposition}
\newtheorem{lem}[thm]{Lemma}
\newtheorem{cor}[thm]{Corollary}
\newtheorem{rmk}[thm]{Remark}
\newtheorem{dfn}[thm]{Definition}

\begin{abstract}
 We study a Lagrangian numerical scheme for solution of a nonlinear drift diffusion equation on an interval.
 The discretization is based on the equation's gradient flow structure with respect to the Wasserstein distance.
 The scheme inherits various properties from the continuous flow, 
 like entropy monotonicity, mass preservation, metric contraction and minimum/ maximum principles. 
 As the main result, we give a proof of convergence in the limit of vanishing mesh size under a CFL-type condition.
 We also present results from numerical experiments.
\end{abstract}

\maketitle

\section{Introduction}
In this paper, we propose and analyze a very particular spatio-temporal discretization of the following 
nonlinear initial-boundary value problem on an interval $I=[a,b]$:
\begin{align}
 \label{eq:adde}
 \partial_t u = \prss(u)_{xx} + (V_x(x)u)_x,\quad u_x(t;a)=u_x(t;b)=0, \quad u(0;x)=u^0(x)\ge0.
\end{align}
Specifically, we are interested in approximating non-negative weak solutions $u:[0,T]\times I\to\setR_{\ge0}$ to \eqref{eq:adde}
on arbitrary time horizonts $T>0$.
Our assumptions are that 
\begin{itemize}
\item the nonlinearity $\prss:\setR_{\ge0}\to\setR$ is continuous, is $C^2$-smooth on $\setR_{>0}$, 
 and satisfies
 \begin{align}
   \label{eq:pressume}
   \prss(0)=0, \quad
   \prss'(r)>0, 
   \quad \lim_{r\downarrow0}\prss'(r) < \infty,
   \quad \lim_{r\to\infty}\prss'(r)=+\infty,
   \quad s\mapsto\prss(1/s)\quad\text{is concave},
 \end{align}
 the prototypical example being the porous medium term $\prss(r)=r^m$ with some $m>1$;
\item the potential $V:I\to\setR$ is $C^2$-smooth with
 \begin{align}
   \label{eq:Vassume}
   V_x(a)=V_x(b)=0. 
 \end{align}
\end{itemize}
Under the given regularity assumptions, there are numerous possibilities to design efficient numerical schemes for solution of \eqref{eq:adde},
e.g., using finite differences.
The particular discretization under consideration here is special insofar
as it is based on the representation of \eqref{eq:adde} as a gradient flow,
and it inherits certain qualitative features of that variational structure; 
see below.

\subsection{Gradient flow structure}
We summarize some basic facts about the variational formulation of \eqref{eq:adde}.
The divergence form in combination with the no-flux boundary conditions and \eqref{eq:Vassume} implies the conservation of mass
\begin{align}
 \label{eq:mass}
 \int_I u(t;x)\dd x = M:= \int_I u^0(x)\dd x > 0  \quad \text{for all $t>0$},
\end{align}
and we shall consider $M$ as some fixed quantity from now on.
The space $\dens\subset L^1(I)$ of bounded and strictly positive densities $u$ with total mass $M$ 
can be endowed with the $L^2$-Wasserstein metric $\wass$;
the definition and elementary properties of this metric are reviewed in Section \ref{sct:prelims} below.
Next, introduce the \emph{energy functional} $\anrj$ for $u\in\dens$ by
\begin{align}
 \label{eq:anrj}
 \anrj(u) = \int_I \phi(u(x))\dd x + \int_I u(x)V(x)\dd x,
\end{align}
where the internal energy potential $\phi:\setR_{\ge0}\to\setR$ is an arbitrary second anti-derivative of $r\mapsto P'(r)/r$;
see \eqref{eq:prsstophi}.

The link between the energy $\anrj$ and equation \eqref{eq:adde} --- which has been rigorously established by Otto \cite{OttPME} --- 
is that solutions to \eqref{eq:adde} form a gradient flow in the energy landscape of $\anrj$ with respect to the metric $\wass$. 
Further, it has been observed by McCann \cite{McCann}  that the functional $\anrj$ is $(-\Lambda)$-convex along geodesics in $\wass$,
with 
\begin{align}
 \label{eq:lambda}
 \Lambda = \max_{x\in I}\big(-V_{xx}(x)\big) \ge 0.
\end{align}
Consequently, the gradient flow is $(-\Lambda)$-contractive.
Some implications are:
\begin{enumerate}
\item The energy $\anrj(u(t))$ is monotonically decreasing in $t$.
\item Two solutions $u$, $v$ diverge at most at an exponential rate of $\Lambda$ in the Wasserstein distance,
 i.e.,
 \begin{align}
   \label{eq:contract}
   \wass(u(t),v(t)) \le \wass(u^0,v^0)e^{\Lambda t}\quad \text{for all $t>0$}.
 \end{align}
\item There is a unique non-negative and mass preserving global solution for measure valued initial conditions,
 i.e., the density $u^0$ in \eqref{eq:adde} can be replaced by an arbitrary non-negative measure on $I$ with mass $M$.
\end{enumerate}
Below, we discuss in which sense these properties are inherited by our discretization.

\subsection{Discretization}
Semi-discretization in time of gradient flow equations like \eqref{eq:adde} 
has become a key tool in existence proofs and for the rigorous derivation of a priori estimates.
The celebrated \emph{minimizing movement} scheme \cite{AGS}
(also referred to as \emph{JKO} \cite{JKO} or simply \emph{implicit Euler} scheme)
works in the situation at hand as follows:
given a time step $\tau>0$, one defines inductively --- starting from $u_\tau^0=u^0$ --- approximations $u_\tau^n$ of $u(n\tau)$
as minimizers in $\dens$ of ``penalized energy functionals'' $\anrj_\tau(\cdot,u_\tau^{n-1})$, 
given by
\begin{align}
 \label{eq:mm}
 \anrj_\tau(u,u_\tau^{n-1}) = \frac1{2\tau}\wass(u,u_\tau^{n-1})^2 + \anrj(u).
\end{align}
Thanks to the $(-\Lambda)$-convexity of $\anrj$, it follows from the theory developed in \cite{AGS}
that the functions $\bar u_\tau:[0,\infty)\to\dens$ obtained by piecewise constant interpolation in time
converge for $\tau\downarrow0$ to the unique weak solution $u:[0,\infty)\to\dens$ of \eqref{eq:adde}.

To obtain a \emph{full} (spatio-temporal) discretization, we perform the minimization of $\anrj_\tau$ not over the entire set $\dens$,
but over a submanifold $\densN$ of finite dimension $(K-1)\in\setN$.
The discretization parameter $\theh=(\xi_0,\xi_1,\ldots,\xi_K)$ is an increasing sequence of numbers $\xi_k\in[0,M]$, 
with $0=\xi_0<\xi_1<\xi_2<\cdots<\xi_K=M$.
The corresponding submanifold $\densN$ consists of piecewise constant density functions $u\in\dens$ of the form
\begin{align*}
 u = \sum_{k=1}^K u_k\indy_{(x_{k-1},x_k]},
\end{align*}
where the end points $x_k$ of the intervals are variable subject to the constraint
\begin{align*}
 a=x_0 < x_1 < \cdots < x_K=b,
\end{align*}
and the positive weights $u_k$ are given in terms of the $x_k$ by
\begin{align*}
 (x_k-x_{k-1})u_k = \delta_k : = \xi_k-\xi_{k-1} > 0.
\end{align*}
One may think of $\delta_1$ to $\delta_K$ as lumps of mass, 
each of which is uniformly distributed on its respective interval $(x_{k-1},x_k]$, 
see Figure \ref{fig:blocks}.
We refer to this discretization as \emph{Lagrangian} scheme.
\medskip

\begin{center}
 \fbox{
   \begin{minipage}{0.95\textwidth}
     Given a discretization $\Delta=(\tau;\theh)$ consisting of a time step $\tau>0$ and a spatial mesh $\theh$,
     and an initial condition $u_\Delta^0\in\densN$,
     define a \emph{discrete solution} $u_\Delta=(u_\Delta^0,u_\Delta^1,\ldots)$ inductively by
     \begin{align}
       \label{eq:dmm}
       u_\Delta^n = \operatorname*{argmin}_{u\in\densN} \anrj_\tau(u,u_\Delta^{n-1}) \quad \text{for $n=1,2,\ldots$}
     \end{align}
   \end{minipage}
 }
\end{center}
\medskip

The seemingly involved definition of the recursion \eqref{eq:dmm}  leads to a simple and practical numerical scheme, 
whose complexity is comparable to that of a standard discretization of \eqref{eq:adde} by finite differences. 
In addition, the Lagrangian nature of the scheme admits a geometric interpretation of the solution
in terms of transportation of mass elements on $I$ along the characteristics $x_k(t)$.
Concerning structure preservation, we summarize some noteworthy features of this approach:
\begin{itemize}
\item The energy $\anrj(u_\Delta^n)$ is monotone in $n$, and all $u_\Delta^n$ are positive and have the same mass $M$.
\item Any two discrete solutions $u_\Delta$ and $v_\Delta$ on the same grid $\Delta$ satisfy the contraction estimate
 \begin{align*}
   \wass\big(u_\Delta^n,v_\Delta^n)
   \le (1-2\Lambda\tau)^{-n/2}\wass\big(u_\Delta^0,v_\Delta^0\big),
 \end{align*}
 which turns into \eqref{eq:contract} in the limit $\tau\downarrow0$. 
 See Section \ref{sct:contract}.
\item The scheme is applicable to arbitrary initial data $u^0\in L^1(I)$ with finite energy.
 In fact, weak convergence of $u_\Delta^0$ to $u^0$ suffices to conclude
 strong convergence of the discrete solution $u_\Delta$ in $L^1([0,T]\times I)$ to the correct weak solution $u$ of \eqref{eq:adde}.
 See Theorem \ref{thm:main} below.
\item Discrete solutions obey a minimum/maximum principle. 
 See Section \ref{sct:minmax}.
\end{itemize}
The choice of $\densN$ originates from an alternative formulation of equation \eqref{eq:adde}.
Namely, $u$ is a (positive and classical) solution to \eqref{eq:adde} iff 
its inverse distribution function $\theX$ --- see Section \ref{sct:prelims} for its definition ---
satisfies the initial-boundary value problem
\begin{align}
 \label{eq:dde}
 \partial_t\theX =  \psi'(\theX_\xi)_\xi - V_x\circ\theX, \quad \theX(t;0)=0,\,\theX(t;M)=1, \quad \theX(0;\xi) = \theX^0(\xi),
\end{align}
where $\psi:\setR_+\to\setR$ is defined by
\begin{align}
 \label{eq:psi}
 \psi(s) = s\phi(s^{-1}) \quad \text{for all $s>0$}.
\end{align}
The variational structure of \eqref{eq:dde} is quite apparent:
solutions $\theX$ to \eqref{eq:dde} are gradient flows of the functional
\begin{align}
 \label{eq:nrj}
 \nrj(\theX) = \int_0^M\psi\big(\theX_\xi(\xi)\big)\dd\xi + \int_0^M V\circ\theX(\xi)\dd\xi
\end{align}
with respect to the usual scalar product on $L^2([0,M])$.
In effect, we discretize the $L^2$-gradient flow \eqref{eq:dde} rather than the $\wass$-gradient flow \eqref{eq:adde},
representing $\theX$ as a linear combination of piecewise linear ansatz functions with respect to the (time-independent) mesh $\theh$.

\subsection{Convergence result}
Our main result is the following.
\begin{thm}
 \label{thm:main}
 Let a non-negative initial condition $u^0\in L^1(I)$ of mass $M$ with $\anrj(u^0)<\infty$ be given,
 and fix a time horizont $T>0$.

 Consider a sequence of discretizations $\Delta^{(j)}=(\tau^{(j)};\theh^{(j)})$,
 consisting of time steps $\tau^{(j)}\downarrow0$ and spatial meshes $\theh^{(j)}$ with $\max_k(\delta^{(j)}_k)\downarrow0$,
 and an associated sequence of initial conditions $u_{\Delta^{(j)}}^0\in\densNintro$.
 Assume that $u_{\Delta^{(j)}}^0\to u^0$ weakly in $L^1(I)$, that $\anrj(u_{\Delta^{(j)}}^0)\le\overline\anrj$,
 that $\max_k\delta^{(j)}_k/\min_\ell\delta^{(j)}_\ell\le\bar\alpha$,
 and that the following \emph{inverse CFL condition} holds:
 \begin{align}
   \label{eq:lfc}
   \big(\max_k\delta^{(j)}_k\big)^2 < 6\psi''\bigg(\frac{6\bar\alpha e^{2\Lambda T}}{\min_x u_{\Delta^{(j)}}^0}\bigg)\tau^{(j)},
 \end{align}
 with $\psi$ defined in \eqref{eq:psi}, and with $\Lambda\ge0$ from \eqref{eq:lambda}.

 The scheme \eqref{eq:dmm} produces a sequence of discrete solutions $u_{\Delta^{(j)}}$.
 Denote by $\bar u_{\Delta^{(j)}}:[0,\infty)\to\densNj$ the respective interpolants that are piecewise constant in time, see \eqref{eq:tinterpolate}.
 Then $\bar u_{\Delta^{(j)}}$ converges strongly in $L^1([0,T]\times I)$ to the unique weak solution $u$ of \eqref{eq:adde}.
\end{thm}
A comment is due on condition \eqref{eq:lfc}.
Since $\psi''(s)\to0$ for $s\to\infty$, this condition implies that the non-negative initial datum $u^0$
needs to be approximated by strictly positive data $u_{\Delta^{(j)}}^0$,
and the smaller one whishes to choose the minimal value of $u_{\Delta^{(j)}}^0$, 
the finer one needs to make the grid $\theh^{(j)}$.
Condition \eqref{eq:lfc} thus quantifies the intuitive requirement that not only the mesh of $\xi_k$'s in $[0,M]$,
but also the induced mesh of $x_k$'s in $I$ should become arbitrarily fine in the limit, uniformly for all times $t\in[0,T]$.
Consequently, our scheme does not allow to track propagating fronts --- like spreading Barenblatt profiles --- directly on the discrete level as in \cite{budd,OttScheme},
but it is able to approximate these fronts arbitrarily well with strictly positive solutions if the mesh is sufficiently fine.

\subsection{Related results from the literature}
Studies on Lagrangian schemes for \eqref{eq:adde} are widely scattered in the literature.
Already MacCamy and Sokolovsky \cite{MacCamy} present a discretization that is almost identical to ours, 
for \eqref{eq:adde} with $\prss(u)=u^2$ and $V\equiv0$.
Another pioneering work in this direction is the paper by Russo \cite{Russo}, 
who compares several (semi-)Lagrangian discretizations in the linear case $\prss(u)=u$;
extensions to two spatial dimensions are also discussed.
Later, Budd et al \cite{budd} used a moving mesh to capture self-similar solutions of the porous medium equation on the whole line.
The general theme was picked up recently by Carrillo and Moll \cite{CarM}, 
who define a Lagrangian discretization of aggregation equations in two space dimensions,
based on the reformulation in terms of evolving diffeomorphisms \cite{Evans}.

The connection between Lagrangian schemes and the gradient flow structure of equation \eqref{eq:adde} was investigated
by Kinderlehrer and Walkington \cite{Kinderlehrer} and in a series of unpublished theses \cite{OttScheme,Leven}. 
In a recent paper by Westdickenberg and Wilkening \cite{WW}, 
a similar scheme for \eqref{eq:adde} is obtained as a by-product in the process of designing 
a structure preserving discretization for the Euler equations.
Burger et al \cite{BCW} devise a numerical scheme for \eqref{eq:adde} in dimension two on basis of the gradient flow structure, 
using the hydrodynamical formulation of the Wasserstein distance \cite{Brenier} instead of the Lagrangian approach.
The Lagrangian approach was adapted to fourth order equations, 
namely by Cavalli and Naldi \cite{Naldi} for the Hele-Shaw flow,
and by D\"uring et al \cite{DMM} for the DLSS equation.

In the aforementioned works, numerical schemes are defined and used in experiments;
qualitative properties and convergence are not studied analytically.
Some analytical investigations have been carried out by Gosse and Toscani \cite{GosT}: 
for a Lagrangian scheme with explicit time discretization, 
they prove comparison principles, and they rigorously discuss stability and consistency.
Also, a full discretization of the Keller-Segel model has been analyzed by Blanchet et al \cite{BCC} in view of convergence to equilibrium.
However, to the best of our knowledge, a proof for convergence of discrete to continuous (weak) solutions is not available in the literature.

Finally, a remark is due on an alternative way of proving convergence of the scheme.
By use of stability results for gradient flows \cite{AGS,ALS} and the machinery of $\Gamma$-convergence,
it seems likely that Theorem \ref{thm:main} can be obtained by exploiting the variational structure more deeply than we do here.
In particular, the theory on perturbed $\lambda$-contractive gradient flows developed by Serfaty \cite{Serfaty} 
indicates an alternative route towards the same goal.
We followed the elementary approach based on a priori estimates here, 
partly in order to avoid heavy machinery,
but mainly with the aim to develop a ``stable'' concept of proof that generalizes more directly 
to gradient flows without convexity properties (like fourth order equations).

\subsection{Outline of the paper}
Section \ref{sct:prelims} below summarizes some basic results on inverse distribution functions and convexity in the Wasserstein metric.
In Section \ref{sct:discretex}, we describe in detail the spatial discretization and study the restrictions of the Wasserstein metric and energy to the $\densN$.
The discrete scheme \eqref{eq:dmm} is studied in Section \ref{sct:discretet}, and we derive the Euler-Lagrange equations.
Section \ref{sct:quality} provides a summary of some qualitative properties of the discretization, 
like metric contraction and the minimum/maximum principle.
The proof of Theorem \ref{thm:main} is given in Section \ref{sct:convergence}.
The paper concludes with the results of various numerical experiments in Section \ref{sct:numerics},
and with a calculation of the consistency order.

\section{Preliminaries and notations}
\label{sct:prelims}
For an introduction to the theory of optimal transportation, we refer to \cite{VilBook}.
A comprehensive theory of gradient flows in the Wasserstein metric can be found in \cite{AGS}.

\subsection{Inverse distribution functions}
Throughout the paper, we shall denote by
\begin{align*}
 \dens := \Big\{ u\in L^1(I)\cap L^\infty(I)\, \Big|\, \essinf_{x\in I}u(x)>0,\ \int_I u(x)\dd x = M \Big\}
\end{align*}
the space of positive density functions of total mass $M$.
For $u\in\dens$, define its distribution function $U:I\to[0,M]$ by
\begin{align*}
 U(t;x) = \int_a^x u(t;y)\dd y,
\end{align*}
and introduce its inverse function $\theX=U^{-1}:[0,M]\to I$.
By our choice of $\dens$, the latter is well-defined and belongs to
\begin{align*}
 \xspc := \big\{ \theX\in C^{0,1}([0,M];I)\,\big|\, \theX(0)=a,\,\theX(M)=b,\,\text{$\theX$ strictly increasing} \big\}.
\end{align*}
Thanks to the Lipschitz continuity of $U$ and $\theX$, we can differentiate the identity $U\circ\theX(\xi)=\xi$ at almost every $\xi\in[0,M]$
and obtain the relation
\begin{align}
 \label{eq:density}
 u(\theX(\xi))\theX_\xi(\xi)=1 \qquad \text{for a.e. $\xi\in[0,M]$}.
\end{align}
The inverse distribution function allows for an explicit representation of the Wasserstein distance in one spatial dimension.
\begin{lem}[see e.g. \cite{VilBook}]
 \label{lem:miracle}
 Let $u_0,u_1\in\dens$ have inverse distribution functions $\theX_0,\theX_1\in\xspc$.
 Then their Wasserstein distance amounts to
 \begin{align}
   \label{eq:miracle}
   \wass(u_0,u_1) = \bigg( \int_0^M [X_1(\xi)-X_0(\xi)]^2\dd\xi \bigg)^{1/2},
 \end{align}
 and a minimal geodesic $(u_s)_{0\le s\le 1}$ connecting $u_0$ to $u_1$ in $\wass$ is given by $X_s=sX_1+(1-s)X_0$.
\end{lem}

\subsection{Properties of the energy}
Given $\prss:\setR_{\ge0}\to\setR$, let $\phi:\setR_{\ge0}\to\setR$ be an arbitrary second anti-derivative of $r\mapsto\prss'(r)/r$,
and define $\psi:\setR_+\to\setR$ by $\psi(s)=s\phi(1/s)$.
Introduce the functionals $\anrj$ on $\dens$ and $\nrj$ on $\xspc$, respectively, by \eqref{eq:anrj} and \eqref{eq:nrj}.
\begin{lem}
 \label{lem:phipsi}
 For every $u\in\dens$ with inverse distribution function $\theX\in\xspc$,
 one has
 \begin{align}
   \label{eq:convertnrj}
   \anrj(u) = \nrj(\theX).
 \end{align}
 Further, $\phi$ is strictly convex and satisfies
 \begin{align}
   \label{eq:prsstophi}
   \prss(r) = r\phi'(r) + \phi(0) - \phi(r) \quad \text{for all $r\ge 0$}.
 \end{align}
 Finally, $\psi(s)\to\infty$ for $s\downarrow0$, 
 and $\psi''$ is a positive non-increasing function.
\end{lem}
\begin{proof}
 We perform the change of variables $x=\theX(\xi)$ under the integrals in the definition \eqref{eq:anrj}
 and use \eqref{eq:density}:
 \begin{align*}
   \anrj(u) 
   &= \int_0^M \phi\big(u(\theX(\xi))\big)\theX_\xi(\xi)\dd\xi 
   + \int_0^M V(\theX(\xi))u(\theX(\xi))\theX_\xi(\xi)\dd\xi \\
   &= \int_0^M \phi\Big(\frac1{\theX_\xi(\xi)}\Big) \theX_\xi(\xi)\dd\xi
   + \int_0^M V(\theX(\xi))\dd\xi = \nrj(\theX).
 \end{align*}
 The claims about $\phi$ and $\psi$ are direct consequences of the hypotheses in \eqref{eq:pressume}.
 By definition, $\phi''(r)=\prss'(r)/r>0$ for all $r>0$, so $\phi$ is strictly convex.
 \eqref{eq:prsstophi} follows by differentiation of both sides w.r.t.\ $r>0$.
 It follows further that
 \begin{align*}
   \psi'(s) &= \phi(s^{-1})-s^{-1}\phi'(s^{-1}) = \phi(0) - \prss(s^{-1}), \\
   \psi''(s) &= s^{-2}\prss'(s^{-1}) > 0, \\
   \psi'''(s) &= - \dd^2\prss(s^{-1})/\dd s^2 \ge 0,
 \end{align*}
 so $\psi''$ is indeed positive and non-increasing.
 Finally,
 \begin{align*}
   \lim_{s\downarrow0}\psi(s) 
   = \psi(1) + \lim_{s\downarrow0}\int_s^1 \prss(\sigma^{-1})\dd\sigma 
   = \psi(1) + \lim_{r\to\infty}\int_1^r \frac{\prss(\rho)}{\rho^2}\dd\rho
   = + \infty
 \end{align*}
 since $\prss'(\rho)\to\infty$ for $\rho\to\infty$, and hence also $\prss(\rho)/\rho\to\infty$.
\end{proof}
The convexity of the functional $\anrj$ with respect to the Wasserstein metric is most conveniently studied
when the latter is considered as a functional of $\theX$ instead of $u$.
Indeed, by Lemma \ref{lem:miracle} above, 
geodesic interpolation between $u_0,u_1\in\dens$ corresponds to linear linterpolation between $\theX_0,\theX_1\in\xspc$.
\begin{lem}
 \label{lem:convex}
 The functional $\nrj$ is bounded from below,
 \begin{align}
   \label{eq:below}
   \nrj(\theX) \ge \nrjbelow := (b-a)\phi\Big(\frac{M}{b-a}\Big) + M\min_{x\in I}V(x),
 \end{align}
 and it is $(-\Lambda)$-convex on $\xspc$ with the $\Lambda$ given in \eqref{eq:lambda},
 i.e.,
 \begin{align}
   \label{eq:convex}
   \nrj\big((1-s)\theX^0+s\theX^1\big) 
   \le (1-s)\nrj(\theX^0) + s\nrj(\theX^1) + \frac{\Lambda s(1-s)}2 \int_0^M [\theX^0(\xi)-\theX^1(\xi)]^2\dd\xi
 \end{align}
 for all $\theX^0,\theX^1\in\xspc$, and every $s\in[0,1]$.
\end{lem}
\begin{proof}
 Since $\phi$ is convex,
 the lower bound follows by Jensen's inequality:
 \begin{align*}
   \nrj(\theX) 
   \ge M\psi\bigg(\int_0^M \theX_\xi(\xi)\frac{\dd\xi}M\bigg) + \int_0^M \min_{x\in I}V(x)\dd\xi.
 \end{align*}
 By definition of $\psi$, this yields \eqref{eq:below}.
 Next, let $\theX^0,\theX^1\in\xspc$ and $s\in[0,1]$ be given.
 Since $\psi:\setR_+\to\setR$ is convex by hypothesis,
 it follows in particular that
 \begin{align*}
   \int_0^M \psi\big((1-s)\theX^0_\xi(\xi)+s\theX^1_\xi(\xi)\big)\dd\xi
   \le (1-s)\int_0^M \psi\big(\theX^0_\xi(\xi)\big)\dd\xi + s \int_0^M \psi\big(\theX^1_\xi(\xi)\big)\dd\xi.
 \end{align*}
 Further, a Taylor expansion yields
 \begin{align*}
   V\big((1-s)y+sz\big) \le (1-s)V(y)+sV(z) +\frac\Lambda2 s(1-s)(y-z)^2
 \end{align*}
 for arbitrary $y,z\in I$.
 In combination, this implies inequality \eqref{eq:convex}.
\end{proof}

\section{Spatial discretization}
\label{sct:discretex}
Inside the space $\xspc$ of inverse distribution functions, we define the finite-dimensional subspace $\xspcN$ of those functions,
which are piecewise affine with respect to a given partition $\theh$ of $[0,M]$ into sub-intervals.
Correspondingly, there is a finite-dimensional submanifold $\densN$ of $\dens$ consisting of those densities,
whose inverse distribution functions belong to $\xspcN$.
Densities in $\densN$ are piecewise constant.
Since we shall work simultaneously in the spaces $\densN$ and $\xspcN$, we need to introduce various notations.

\subsection{Ansatz spaces}
\label{sct:prediscrete}
\begin{figure}
 \centering
 \includegraphics[width=0.35\linewidth]{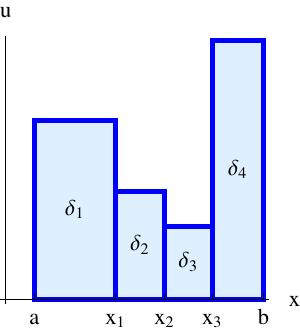}
 \hspace{0.1\linewidth}
 \includegraphics[width=0.35\linewidth]{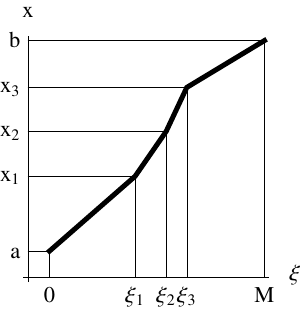}
 \caption{A typical density function $u\in\densN$ (left) and inverse distribution function $\theX\in\xspc$.}
 \label{fig:blocks}
\end{figure}
A vector $\theh=(\xi_0,\xi_1,\ldots,\xi_K)$ with entries $\xi_j$ such that 
\begin{align*}
 0=\xi_0 < \xi_1 < \cdots < \xi_K = M
\end{align*}
defines a partition of $[0,M]$ into $K$ sub-intervals.
We denote the lengths of the intervals by
\begin{align*}
 \delta_k = \xi_k-\xi_{k-1} \quad \text{for all $k=1,\ldots,K$},
\end{align*}
and introduce further
\begin{align}
 \label{eq:deltamin}
 \delmin(\theh) = \min_k\delta_k, \quad
 \delmax(\theh) = \max_k\delta_k, \quad
 \alpha(\theh) = \frac{\delmax(\theh)}{\delmin(\theh)}.
\end{align}
The associated $(K-1)$-dimensional reduction of the space $\xspc$ is given by
\begin{align*}
 \xspcN := \big\{ \theX\in\xspc\,\big| \, \text{$\theX$ piecewise affine on each $[\xi_{k-1},\xi_k]$, for $k=1,\ldots,K$}\big\}.
\end{align*}
Functions in $\xspcN$ are conveniently represented as linear combinations of
the $K+1$ hat functions $\hatf_0$ to $\hatf_K$ defined by
\begin{align*}
 \hatf_m(\xi) =
 \begin{cases}
   (\xi-\xi_{m-1})/\delta_m & \text{for $\xi_{m-1}\le\xi\le\xi_m$ (if $m\ge1$)}, \\
   (\xi_{m+1}-\xi)/\delta_{m+1} & \text{for $\xi_m\le\xi\le\xi_{m+1}$ (if $m\le K-1$)}, \\
   0 & \text{otherwise}.
 \end{cases}
\end{align*}
More precisely,
there is a one-to-one correspondence between functions $\theX\in\xspcN$ 
and vectors in
\begin{align*}
 \xseq := \big\{ \xvec=(x_1,\ldots,x_{K-1}) \,\big| \, a<x_1<x_2<\cdots<x_{K-1}<b \big\}\subset I^{K-1};
\end{align*}
this correspondence is established by means of $\convX_\theh:\xseq\to\xspcN$,
with
\begin{align}
 \label{eq:Xbyx}
 \theX = \convX_\theh[\xvec] = \sum_{k=0}^K x_k\hatf_k .
\end{align}
\begin{rmk}
 By definition, $\xvec\in\xseq$ has components $x_1$ to $x_{K-1}$.
 In \eqref{eq:Xbyx} above and for the rest of the paper,
 we shall always use the convention that
 \begin{align}
   \label{eq:convention}
   x_0 = a \quad\text{and}\quad x_K = b.
 \end{align}
 Also, we introduce in analogy to $\delmax(\theh)$ the mesh width
 \begin{align}
   \label{eq:delmaxx}
   \delmax(\xvec) = \max_k(x_k-x_{k-1}).
 \end{align}
\end{rmk}
Occasionally, it will be more convenient to work with vectors $\gvec \in \setR_+^K$ of difference quotients:
define $\convg_\theh:\xseq\to\setR_+^K$ by
\begin{align}
 \label{eq:g}
 \gvec = (z_1,\ldots,z_K) = \convg_\theh[\xvec]\in\setR_+^K, \quad\text{with}\quad z_k = \frac{x_k-x_{k-1}}{\delta_k},
\end{align}
using again our convention \eqref{eq:convention}.

Finally, we introduce the associated $(K-1)$-dimensional submanifold $\densN := \convf_\theh[\xseq]\subset\dens$
as the image of the injective map $\convf_\theh:\xseq\to\densN$
with
\begin{align}
 \label{eq:1to1}
 \convf_\theh[\xvec] = \sum_{k=1}^Ku_k\indy_{(x_{k-1},x_k]}, \quad\text{where} \quad u_k=\frac{\xi_k-\xi_{k-1}}{x_k-x_{k-1}}.
\end{align}

\subsection{Representation of the Wasserstein distance}
The Wasserstein distance between any two elements of $\densN$ is easy to compute using \eqref{eq:miracle}.
\begin{lem}
 \label{lem:wasserstein}
 Fix a discretization $\theh$, 
 and let $u^0,u^1\in\densN$ have representations $u^0 = \convf_\theh[\xvec^0]$, $u^1 = \convf_\theh[\xvec^1]$, 
 with $\xvec^0,\xvec^1\in\xseq$.
 Then 
 \begin{align}
   \label{eq:anothermiracle}
   \wass(u^0,u^1)^2 = (\xvec^0-\xvec^1)^T\Wmat(\xvec^0-\xvec^1)
 \end{align}
 with the symmetric tridiagonal matrix $\Wmat = (\Wmat_{m,k})_{m,k=1}^{K-1}\in\setR^{(K-1)\times(K-1)}$
 given by
 \begin{align}
   \label{eq:Wmat}
   \Wmat = \frac16
   \begin{pmatrix}
     2(\delta_1+\delta_2) &\delta_2  & \cdots & 0 & 0 
     \\ 
     \delta_2 & 2(\delta_2+\delta_3) & \cdots & 0 & 0 
     \\  
     \vdots & \vdots & \ddots & \vdots & \vdots 
     \\ 
     0 & 0 & \cdots  & 2(\delta_{K-2}+\delta_{K-1}) & \delta_{K-1} 
     \\ 
     0 & 0 & \cdots  & \delta_{K-1} & 2(\delta_{K-1}+\delta_K) 
   \end{pmatrix}.
 \end{align}
 Moreover, for every $v\in\setR^{K-1}$,
 \begin{align}
   \label{eq:posdef}
   \frac16\sum_{k=1}^{K-1}(\delta_k+\delta_{k+1})v_k^2
   \le v^T\Wmat v \le
   \frac12\sum_{k=1}^{K-1}(\delta_k+\delta_{k+1})v_k^2.    
 \end{align}
\end{lem}
\begin{proof}
 Plugging the representation \eqref{eq:Xbyx} into the definition \eqref{eq:miracle} of the Wasserstein distance
 yields
 \begin{align*}
   \wass(u^0,u^1)^2
   = \sum_{m,k=0}^K \bigg([\xvec^0-\xvec^1]_m [\xvec^0-\xvec^1]_k \int_0^M \hatf_m(\xi)\hatf_k(\xi)\dd\xi\bigg).
 \end{align*}
 Since $x^0_0=x^1_0=a$ and $x^0_K=x^1_K=b$, the last sum actually runs only over indices from one to $K-1$.
 Therefore, in order to prove \eqref{eq:anothermiracle},
 it suffices to show that
 \begin{align*}
   \Wmat_{m,k} = \int_0^M \hatf_m(\xi)\hatf_k(\xi)\dd\xi
 \end{align*}
 for all $m,k=1,\ldots,K-1$.
 Since $\hatf_m$ has support $[\xi_{m-1},\xi_{m+1}]$ by definition,
 it follows that $\Wmat_{m,k}=0$ if $|m-k|\ge2$.
 Moreover, for $m=k$ we have
 \begin{align*}
   \int_0^M \hatf_m(\xi)^2\dd\xi
   &=\delta_m^{-2}\int_{\xi_{m-1}}^{\xi_m} (\xi-\xi_{m-1})^2\dd\xi 
   + \delta_{m+1}^{-2}\int_{\xi_m}^{\xi_{m+1}} (\xi_{m+1}-\xi)^2\dd\xi \\
   & = \delta_m\int_0^1\eta^2\dd\eta + \delta_{m+1}\int_0^1\zeta^2\dd\zeta
   = \frac13(\delta_m+\delta_{m+1}),
 \end{align*}
 and for $k=m+1$,
 \begin{align*}
   \int_0^M \hatf_m(\xi)\hatf_{m+1}(\xi)\dd\xi
   = \delta_m^{-2}\int_{\xi_m}^{\xi_{m+1}} (\xi_{m+1}-\xi)(\xi-\xi_m)\dd\xi
   = \delta_m\int_0^1(1-\eta)\eta\dd\eta
   = \frac16\delta_m.
 \end{align*}
 Finally, let $v\in\setR^{K-1}$ be given and observe that
 \begin{align*}
   3v^T\Wmat v
   &= \sum_{m=1}^{K-1} (\delta_m+\delta_{m+1})v_m^2 + \sum_{m=2}^{K-1} \delta_m v_mv_{m-1} \\
   &= \delta_1v_1^2+\delta_Kv_{K-1}^2 + \sum_{m=2}^{K-1} \delta_m (v_m^2+v_{m-1}^2) + \sum_{m=1}^K\delta_mv_mv_{m-1}.
 \end{align*}
 From here, \eqref{eq:posdef} is immediately deduced using binomial formulas.
\end{proof}

\subsection{Representation of the energy}
The restriction of the energy $\nrj$ from \eqref{eq:convertnrj} to the subspace $\xspc_\theh$ is naturally associated to
the functional $\nrj_\theh:\xseq\to\setR$ with
\begin{align*}
 \nrj_\theh(\xvec) := \nrj(\convX_\theh[\xvec]) = \anrj(\convf_\theh[\xvec]).
\end{align*}
By straight-forward calculations, one obtains the following more explicit representation.
\begin{lem}
 For every $\xvec\in\xseq$, we have
 \begin{align}
   \label{eq:nrjexpl}
   \nrj_\theh(\xvec)
   = \sum_{k=1}^K \delta_k\psi\bigg(\frac{x_k-x_{k-1}}{\delta_k}\bigg)
   + \int_0^M V \big(\convX_\theh[\xvec](\xi)\big)\dd\xi.
 \end{align}
 Moreover, the (Euclidean) gradient vector $\grd\nrj_\theh(\xvec)=\big(\partial_{x_m}\nrj_\theh(\xvec)\big)_{m=1}^{K-1}\in\setR^{K-1}$ 
 is given by
 \begin{align}
   \label{eq:gradient}
   \big[\grd\nrj_\theh(\xvec)\big]_m 
   &= - \psi'\bigg(\frac{x_{m+1}-x_m}{\delta_{m+1}}\bigg) + \psi'\bigg(\frac{x_m-x_{m-1}}{\delta_m}\bigg)
   + \int_0^M V_x\big(\convX_\theh[\xvec](\xi)\big)\hatf_m(\xi)\dd\xi,
 \end{align}
 and the Hessian matrix $\hess\nrj_\theh(\xvec)=\big(\partial_{x_m x_k}\nrj_\theh(\xvec)\big)_{m,k=1}^{K-1}\in\setR^{(K-1)\times(K-1)}$ is symmetric with
 \begin{align}
   \label{eq:hess}
   \big[\hess\nrj_\theh(\xvec)\big]_{m,k}
   = \begin{cases}
     \displaystyle{\frac1{\delta_{m+1}}\psi''\bigg(\frac{x_{m+1}-x_m}{\delta_{m+1}}\bigg) + \frac1{\delta_m}\psi''\bigg(\frac{x_m-x_{m-1}}{\delta_m}\bigg)} \\
     \hfill\displaystyle{+ \int_0^M V_{xx}\big(\convX_\theh[\xvec]\big)\hatf_m^2\dd\xi} \quad \text{if $m=k$}, \medskip \\
     \displaystyle{-\frac1{\delta_m}\psi''\bigg(\frac{x_m-x_{m-1}}{\delta_m}\bigg)
      + \int_0^M V_{xx}\big(\convX_\theh[\xvec]\big)\hatf_m\hatf_{m-1}\dd\xi} \quad \text{if $k=m-1$}, \\
    0 \hfill \text{if $1\le k<m-1$}.
    \end{cases}
  \end{align}
\end{lem}
Further, the functional $\nrj_\theh$ inherits boundedness and convexity from $\nrj$.
\begin{lem}
 $\nrj_\theh$ is bounded from below by $\nrjbelow$ defined in \eqref{eq:below}.
 Further, it is $(-\Lambda)$-convex with respect to the quadratic structure induced by $\Wmat$, i.e.,
 $\nabla^2\nrj_\theh(\xvec) + \Lambda\Wmat$ is positive semi-definite for arbitrary $\xvec^0\in\xseq$.
 Consequently,
 \begin{align}
   \label{eq:convex3}
   (\xvec^1-\xvec^0)^T\big(\grd\nrj_\theh(\xvec^1)-\grd\nrj_\theh(\xvec^0)\big)
   \ge -\Lambda(\xvec^1-\xvec^0)^T\Wmat(\xvec^1-\xvec^0)
 \end{align}
 holds for every $\xvec^0,\xvec^1\in\xseq$.
\end{lem}
\begin{proof}
 Boundedness from below is a trivial consequence of \eqref{eq:below} and the definition of $\nrj_\theh$ by restriction of $\nrj$.
 Convexity is a direct consequence of the convexity \eqref{eq:convex} of $\nrj$,
 taking into account \eqref{eq:anothermiracle}, and that $\convX_\theh$ is an affine map.
 The estimate \eqref{eq:convex3} is obtained by Taylor expansion.
\end{proof}

\section{Time-discrete evolution}
\label{sct:discretet}
Throughout this section, we fix a pair $\Delta=(\tau,\theh)$
of a time step with $\tau>0$ and a spatial discretization $\theh=(\xi_0,\ldots,\xi_K)$.

\subsection{Minimizing movements}
With the finite-dimensional manifold $\densN$ given at the end of Section \ref{sct:prediscrete} above,
the procedure \eqref{eq:dmm} can now be used to define inductively --- starting from a prescribed initial datum $u^0_\Delta\in\densN$ ---
a \emph{discrete solution} $u_\Delta:=(u_\Delta^n)_{n=0}^\infty$.
\begin{prp}
 \label{prp:existence}
 Assume that $\tau\Lambda<1$, with $\Lambda\ge0$ defined in \eqref{eq:lambda}.
 Recall the definition of $\anrj_\tau$ from \eqref{eq:mm}.
 Then, for every $u_\Delta^0\in\densN$, there is a sequence $(u_\Delta^n)_{n=0}^\infty$,
 such that $u_\Delta^n\in\densN$ is the unique minimizer of $\anrj_\tau(\cdot,u_\Delta^{n-1})$ on the restricted set $\densN$, for every $n\in\setN$.
 \smallskip

 \noindent
 Moreover, define the associated sequence $(\xvec_\Delta^n)_{n=0}^\infty$ of $\xvec_\Delta^n\in\xseq$ by
 \begin{align}
   \label{eq:ufromx}
   u_\Delta^n = \convf_\theh[\xvec_\Delta^n].
 \end{align}
 Then each $\xvec_\Delta^n$ is the unique solution $\xvec\in\xseq$ to the system Euler-Lagrange equations
 \begin{align}
   \label{eq:el}
   \frac1\tau \Wmat(\xvec-\xvec_\Delta^{n-1}) = - \grd\nrj_\theh(\xvec),
 \end{align}
 with $\grd\nrj_\theh(\xvec)$ explicitly given in \eqref{eq:gradient}.
\end{prp}
\begin{proof}
 By definition of $\densN$ as the image of $\xseq$ under $\convf_\theh$,
 it suffices to prove unique solvability of the minimization problems
 \begin{align*}
   \nrj_\Delta(\xvec,\xvec_\Delta^{n-1}) := \frac1{2\tau}(\xvec-\xvec_\Delta^{n-1})^T\Wmat (\xvec-\xvec_\Delta^{n-1}) + \nrj_\Delta(\xvec)
   \quad\to\quad \min
 \end{align*}
 for $\xvec\in\xseq$.
 To this end, observe that
 \begin{align*}
   \nrj_\Delta(\xvec,\xvec_\Delta^{n-1})
   = \nrj_\theh(\xvec)+\frac\Lambda2(\xvec-\xvec_\Delta^{n-1})^T\Wmat(\xvec-\xvec_\Delta^{n-1})
   + \frac12(\tau^{-1}-\Lambda)(\xvec-\xvec_\Delta^{n-1})^T\Wmat(\xvec-\xvec_\Delta^{n-1})
 \end{align*}
 for every $\xvec\in\xseq$.
 From Lemma \ref{lem:convex}, we know that the sum of the first two terms on the right-hand side constitutes a convex function in $\xvec\in\xseq$.
 Since $\tau\Lambda<1$, and since $\Wmat$ is positive definite by Lemma \ref{lem:wasserstein},
 the last term is strictly convex.
 Thus, $\nrj_\Delta(\cdot,\xvec_\Delta^{n-1})$ possesses at most one critical point in $\xseq$.

 To show the existence of a minimizer, let $(\xvec^{(j)})_{j\in\setN}$ be a minimizing sequence for $\nrj_\Delta(\cdot,\xvec^{n-1})$ in $\xseq_\theh$.
 Since each of the $K-1$ components $x^{(j)}_k$ belongs to the compact interval $I$,
 we may assume without loss of generality that $\xvec^{(j)}$ converges to some $\xvec^*\in I^{K-1}$.
 It remains to be proven that $\xvec^*\in\xseq_\theh$.
 Since $(\xvec^{(j)})_{j\in\setN}$ is a minimizing sequence, $\nrj_\Delta(\xvec^{(j)},\xvec^{n-1})$ is bounded,
 and so, for every $m\in\{1,\ldots,K\}$:
 \begin{align*}
   C
   & \ge \frac1{2\tau}(\xvec^{(j)}-\xvec_\Delta^{n-1})^T\Wmat(\xvec^{(j)}-\xvec_\Delta^{n-1}) 
   + \sum_{k=1}^K\delta_k\psi\bigg(\frac{x^{(j)}_k-x^{(j)}_{k-1}}{\delta_k}\bigg)
   + \int_0^M V\big(\convX_\theh[\xvec^{(j)}](\xi)\big)\dd\xi \\
   &\ge \delta_m\psi\bigg(\frac{x^{(j)}_m-x^{(j)}_{m-1}}{\delta_m}\bigg) 
   + (M-\delta_m)\psi\Big(\frac{b-a}{M-\delta_m}\Big) + M \min_{x\in I}V(x).
 \end{align*}
 Since $\psi(s)\to\infty$ for $s\downarrow0$, 
 this implies that $x^n_m-x^n_{m-1}\ge\epsilon\delta_m >0$ with some $\epsilon>0$ for all $n\in\setN$,
 and thus also $x^*_m-x^*_{m-1}\ge\epsilon\delta_m>0$, implying $\xvec^*\in\xseq$.
 By continuity of $\nrj_\Delta(\cdot,\xvec_\Delta^{n-1})$ on $\xseq$, 
 it follows that $\xvec^*$ is a minimizer.

 The argument shows that $\nrj_\Delta(\cdot,\xvec_\Delta^{n-1})$ possesses a unique critical point in $\xseq$,
 thus the corresponding Euler-Lagrange equations \eqref{eq:el} are uniquely solvable.
\end{proof}

\subsection{Euler-Lagrange equations for the difference quotients}
The analysis that follows will be based primarily on another representation of the system \eqref{eq:el} of Euler-Lagrange equations,
which is formulated in terms of the difference quotients $\gvec_\Delta^n=(z_1^n,\ldots,z_K^n)$ introduced in \eqref{eq:g}:
\begin{align}
 \label{eq:zfromx}
 z_k^n = \frac{x^n_k-x^n_{k-1}}{\delta_k} = \frac1{u^n_k}.
\end{align}
To begin with, we introduce quadratic analogues $\hatF_1,\ldots,\hatF_K:[0,M]\to\setR$ of the piecewise linear hat functions $\hatf_0,\ldots,\hatf_K$ as follows.
Let the numbers $\gamma_1,\ldots,\gamma_K\in(-1,1)$ be defined by $\gamma_1=\gamma_K=0$, and
\begin{align}
 \label{eq:gamma}
 \gamma_k = \frac{\delta_{k+1}-\delta_{k-1}}{\delta_{k+1}+2\delta_k+\delta_{k-1}} \quad\text{for $k=2,\ldots,K-1$}.
\end{align}
Then the $\hatF_k$ are given by
\begin{align*}
 \hatF_k(\xi) =
 \begin{cases}
   \frac{1+\gamma_k}{2\delta_{k-1}}(\xi-\xi_{k-2})^2 & \text{if $\xi_{k-2}\le\xi\le\xi_{k-1}$}, \\
   \frac{1-\gamma_k^2}4(\delta_{k+1}+\delta_k+\delta_{k-1})
   - \frac1{4\delta_k}\big(2\xi-(\xi_k+\xi_{k-1})-\gamma_k\delta_k\big)^2
   & \text{if $\xi_{k-1}\le\xi\le\xi_k$}, \\
   \frac{1-\gamma_k}{2\delta_{k+1}}(\xi_{k+1}-\xi)^2 & \text{if $\xi_k\le\xi\le\xi_{k+1}$}, \\
   0 & \text{otherwise}
 \end{cases}
\end{align*}
for $k=2,\ldots,K-1$, and by
\begin{align*}
 \hatF_1(\xi) &=
 \begin{cases}
   \frac12(\delta_1+\delta_2)-\frac1{2\delta_1}\xi^2 & \text{if $0\le\xi\le\xi_1$}, \\
   \frac1{2\delta_2}(\xi_2-\xi)^2 & \text{if $\xi_1\le\xi\le\xi_2$}, \\
   0 & \text{otherwise}, 
 \end{cases}
 \\
 \hatF_K(\xi) &=
 \begin{cases}
   \frac1{2\delta_{K-1}}(\xi-\xi_{K-2})^2 & \text{if $\xi_{K-2}\le\xi\le\xi_{K-1}$}, \\
   \frac12(\delta_K+\delta_{K-1})-\frac1{2\delta_K}(M-\xi)^2 & \text{if $\xi_{K-1}\le\xi\le M$}, \\
   0 & \text{otherwise}.
 \end{cases}
\end{align*}
\begin{lem}
 For each $k=2,\ldots,K-1$, the function $\hatF_k$ is supported on $[\xi_{k-2},\xi_{k+1}]$ and satisfies
 \begin{align}
   \label{eq:bighat}
   - (\hatF_k)_\xi = (1-\gamma_k)\hatf_k-(1+\gamma_k)\hatf_{k-1},
 \end{align}
 and we have $(\hatF_1)_\xi=-\hatf_1$ and $(\hatF_K)_\xi=\hatf_{K-1}$.
\end{lem}
\begin{proof}
 This follows directly from the definition.
\end{proof}
Next, we define the matrix $\Wnew=(\Wnew_{m,m'})_{m,m'=1}^K\in\setR^{K\times K}$ by
\begin{align}
 \label{eq:3}
 \Wnew_{m,k} = \int_{\xi_{k-1}}^{\xi_k} \hatF_m(\xi)\dd\xi.
\end{align}
The matrix $\Wnew$ essentially plays the same role for the $\gvec_\Delta^n$ as $\Wmat$ for the $\xvec_\Delta^n$.
Its entries are more complicated, but still can be calculated explicitly.
\begin{lem}
 The matrix $\Wnew$ is tri-diagonal and has entries
 \begin{align}
   \label{eq:wnew}
   \Wnew_{m,k} =
   \begin{cases}
     \frac16\delta_m^2+\frac{1-\gamma_m}4\delta_m\delta_{m+1}+\frac{1+\gamma_m}4\delta_m\delta_{m-1} & \text{if $2\le k=m\le K-1$}, \\
     \frac13\delta_1^2+\frac12\delta_2\delta_1 & \text{if $k=m=1$}, \\
     \frac13\delta_K^2+\frac12\delta_{K-1}\delta_K & \text{if $k=m=K$}, \\
     \frac{1+\gamma_m}6\delta_{m-1}^2 & \text{if $k=m-1$}, \\
     \frac{1-\gamma_m}6\delta_{m+1}^2 & \text{if $k=m+1$}, \\
     0 & \text{otherwise}.
   \end{cases}
 \end{align}
\end{lem}
\begin{proof}
 The explicit representation of $\Wnew$ is obtained by a tedious, but straight-forward computation that 
 requires nothing but integration of quadratic polynomials and the use of \eqref{eq:gamma}.
\end{proof}
\begin{lem}\label{lem:gamma}
 For the solution $(\xvec_\Delta^n)_{n=0}^\infty$ obtained in Proposition \ref{prp:existence}, 
 let $(\gvec_\Delta^n)_{n=0}^\infty$ be the associated sequence from \eqref{eq:zfromx}.
 Then each $\gvec_\Delta^n$ satisfies the following system of Euler-Lagrange equations:
 \begin{align}
   \label{eq:el3}
   \begin{split}
     \frac1\tau\big[\Wnew(\gvec_\Delta^n-\gvec_\Delta^{n-1})\big]_m
     &= (1-\gamma_m)\psi'(z^n_{m+1}) - 2\psi'(z^n_m) + (1+\gamma_m)\psi'(z^n_{m-1}) \\
     & \qquad -\sum_{k=m-1}^{m+1} z^n_k\int_{\xi_{k-1}}^{\xi_k}V_{xx}\big(\convX_\theh[\xvec^n]\big)\hatF_m\dd\xi
   \end{split}
 \end{align}
 for every $m=2,\ldots,K-1$,
 and
 \begin{align}
   \label{eq:el2a}
   \frac1\tau[\Wnew(\gvec_\Delta^n-\gvec_\Delta^{n-1})]_1
   &= \psi'(z^n_2)-\psi'(z^n_1) -\sum_{k=1}^2 z^n_k\int_{\xi_{k-1}}^{\xi_k}V_{xx}\big(\convX_\theh[\xvec^n]\big)\hatF_1\dd\xi, \\
   \label{eq:el2b}
   \frac1\tau[\Wnew(\gvec_\Delta^n-\gvec_\Delta^{n-1})]_K
   &= \psi'(z^n_{K-1})-\psi'(z^n_K) -\sum_{k=K-1}^K z^n_k\int_{\xi_{k-1}}^{\xi_k}V_{xx}\big(\convX_\theh[\xvec^n]\big)\hatF_K\dd\xi.
 \end{align}
\end{lem}
\begin{proof}
 Fix an index $m\in\{2,\ldots,K-1\}$.
 With $\gamma_m$ given by \eqref{eq:gamma}, 
 multiply the $m$th and the $(m-1)$th component of the Euler-Lagrange system \eqref{eq:el} by $1-\gamma_m$ and $1+\gamma_m$, respectively,
 and substract the latter from the first.
 This yields
 \begin{align*}
   &\frac{1}{\tau}\Big(\frac{1-\gamma_m}6\delta_{m+1}x^n_{m+1} 
   + \Big[\frac{1-\gamma_m}3(\delta_m+\delta_{m+1})-\frac{1+\gamma_m}6\delta_m\Big]x^n_m \\
   & \qquad + \Big[\frac{1-\gamma_m}6 \delta_{m-1}-\frac{1+\gamma_m}3(\delta_m+\delta_{m-1})\Big]x^n_{m-1}
   - \frac{1+\gamma_m}6\delta_{m-1}x^n_{m-2}\Big) \\
   &= (1-\gamma_m)\psi'(z^n_{m+1}) - 2\psi'(z^n_m) + (1+\gamma_m)\psi'(z^n_{m-1}) \\
   & \qquad - \int_0^M V_x\big(\convX_\theh[\xvec^n]\big)\big[(1-\gamma_m)\hatf_m-(1+\gamma_m)\hatf_{m-1}\big]\dd\xi.
 \end{align*}
 The expression on the left-hand side can be rewritten as
 \begin{align*}
   &\frac{1-\gamma_m}6\delta_{m+1}(x^n_{m+1}-x^n_m)
   +\Big[\frac{1-\gamma_m}2\delta_{m+1}+\frac{1-3\gamma_m}6\delta_m\Big]x^n_m \\
   & \qquad -\Big[\frac{1+\gamma_m}2\delta_{m-1}+\frac{1+3\gamma_m}6\delta_m\Big]x^n_{m-1}
   +\frac{1+\gamma_m}6\delta_{m-1}(x^n_{m-1}-x^n_{m-1}) \\
   & = \Wnew_{m,m+1}z^n_{m+1} + \Wnew_{m,m}z^n_m + \Wnew_{m,m-1}z^n_{m-1},
 \end{align*}
 where we have used the relation
 \begin{align*}
   \frac{1-\gamma_m}2\delta_{m+1}+\frac{1-3\gamma_m}6\delta_m
   =\frac{1+\gamma_m}2\delta_{m-1}+\frac{1+3\gamma_m}6\delta_m
   =\frac{1}{\delta_m}\Wnew_{m,m},
 \end{align*}
 which is a consequence of our definition of $\gamma_m$ in \eqref{eq:gamma}.
 Thus, we obtain
 \begin{align}
   \label{eq:el2}
   \begin{split}
     \frac1\tau[\Wnew(\gvec_\Delta^n-\gvec_\Delta^{n-1})]_m 
     &= (1-\gamma_m)\psi'(z^n_{m+1}) - 2\psi'(z^n_m) + (1+\gamma_m)\psi'(z^n_{m-1}) \\
     & \qquad - \int_0^M V_x\big(\convX_\theh[\xvec_\Delta^n]\big)\big[(1-\gamma_m)\hatf_m-(1+\gamma_m)\hatf_{m-1}\big]\dd\xi.
   \end{split}
 \end{align}
 Using property \eqref{eq:bighat} of the function $\hatF_m$, and integrating by parts, we arrive at
 \begin{align}
   \label{eq:cleveribp}
   \begin{split}
     &- \int_0^M V_x\big(\convX_\theh[\xvec_\Delta^n]\big)\big[(1-\gamma_m)\hatf_m-(1+\gamma_m)\hatf_{m-1}\big]\dd\xi
     = \int_0^M V_x\big(\convX_\theh[\xvec_\Delta^n]\big)(\hatF_m)_\xi\dd\xi \\
     &\qquad = V_x\big(\convX_\theh[\xvec_\Delta^n]\big)(\hatF_m)_\xi\Big|_{\xi=0}^{\xi=M} 
     -\int_0^M \big(\convX_\theh[\xvec_\Delta^n]\big)_\xi V_{xx}\big(\convX_\theh[\xvec_\Delta^n]\big) \hatF_m\dd\xi.
   \end{split}
 \end{align}
 The boundary terms in the second line vanish, since $\hatF_m(0)=\hatF_m(M)=0$.
 Finally, observe that $\big(\convX_\theh[\xvec_\Delta^n]\big)_\xi(\xi)=z_k$ for all $\xi_{k-1}<\xi<\xi_k$. 
 Insert the result into \eqref{eq:el2} to arrive at \eqref{eq:el3}, 

 Equations \eqref{eq:el2a} and \eqref{eq:el2b} are directly obtained from \eqref{eq:el} for $m=1$ and for $m=K-1$, respectively,
 after an integration by parts like in \eqref{eq:cleveribp}.
 The boundary term vanishes because of hypothesis \eqref{eq:Vassume}.
\end{proof}

\subsection{Energy dissipation}
The estimates derived below are at the core of our convergence proof in Section \ref{sct:convergence}.
We start with two energy-type estimates, which are classical in the theory of gradient flows.
We recall that $u_\Delta^n=\convf_\theh[\xvec_\Delta^n]$, 
and that the gradient $\grd\nrj_\theh$ is explicitly given in \eqref{eq:gradient}. 
\begin{lem}
 For every $N\in\setN$,
 \begin{align}
   \label{eq:we}
   \frac1{2\tau} \sum_{n=1}^N \wass\big(u_\Delta^n,u_\Delta^{n-1}\big)^2
   \le \anrj_\theh(u_\Delta^0) -\anrj_\theh(u_\Delta^N),    \\
   \label{eq:ee}
   \frac\tau2 \sum_{n=1}^N [\grd\nrj_\theh(\xvec_\Delta^n)]^T\Wmat^{-1}[\grd\nrj_\theh(\xvec_\Delta^n)]
   \le \anrj_\theh(u_\Delta^0) -\anrj_\theh(u_\Delta^N).
 \end{align}
\end{lem}
\begin{proof}
 Since $u_\Delta^n$ minimizes $\anrj_\Delta(\cdot,u_\Delta^{n-1})$,
 we have in particular $\anrj_\Delta(u_\Delta^n,u_\Delta^{n-1})\le \anrj_\Delta(u_\Delta^{n-1},u_\Delta^{n-1})$,
 which implies that
 \begin{align*}
   \frac1{2\tau}\wass(u_\Delta^n,u_\Delta^{n-1})^2 \le \anrj(u_\Delta^{n-1})-\anrj(u_\Delta^{n}).
 \end{align*}
 Evaluation of the telescopic sum yields \eqref{eq:we}.
 To obtain \eqref{eq:ee},
 multiply the system \eqref{eq:el} of Euler-Lagrange equations by $(\tau/2)^{1/2}\Wmat^{-1/2}$,
 take the Euclidean norm on both sides, and sum over $n=1$ to $n=N$:
 \begin{align*}
   \frac\tau2 \sum_{n=1}^N [\grd\nrj_\theh(\xvec_\Delta^n)]^T\Wmat^{-1}[\grd\nrj_\theh(\xvec_\Delta^n)]
   = \sum_{n=1}^N\bigg(\frac1{2\tau}(\xvec_\Delta^n-\xvec_\Delta^{n-1})^T\Wmat (\xvec_\Delta^n-\xvec_\Delta^{n-1})\bigg).
 \end{align*}
 Insert this in \eqref{eq:we} to obtain \eqref{eq:ee}.
\end{proof}
\begin{prp}
 For every $N\in\setN$,
 \begin{align}
   \label{eq:ei}
   \tau\sum_{n=1}^N \sum_{k=1}^{K-1}\frac{\big(\psi'(z^n_{k+1})-\psi'(z^n_k)\big)^2}{\delta_k+\delta_{k+1}}
   \le \alpha(\theh)\, \big[\anrj(u_\Delta^0)-\anrj(u_\Delta^N)+MT \sup_{x\in I}\big(V_x(x)^2\big)\big],
 \end{align}
 with $T=N\tau$ and the ratio $\alpha(\theh)$ being defined in \eqref{eq:deltamin}.
\end{prp}
\begin{proof}
 From the upper estimate on $\Wmat$ in \eqref{eq:posdef},
 it follows that
 \begin{align*}
   [\grd\nrj_\theh(\xvec_\Delta^n)]^T&\Wmat^{-1}[\grd\nrj_\theh(\xvec_\Delta^n)]
   \ge \delmax(\theh)^{-1}[\grd\nrj_\theh(\xvec_\Delta^n)]^T [\grd\nrj_\theh(\xvec_\Delta^n)] \\
   &\ge \delmax(\theh)^{-1}\sum_{k=1}^{K-1}\bigg[\psi'(z^n_{k+1})-\psi'(z^n_k) - \int_0^M V_x\big(\convX_\theh[\xvec^n]\big)\hatf_k\dd\xi\bigg]^2 \\
   &\ge \delmax(\theh)^{-1}\sum_{k=1}^{K-1}
   \bigg[\frac12\big(\psi'(z^n_{k+1})-\psi'(z^n_k)\big)^2 
   - \bigg(\int_0^M V_x\big(\convX_\theh[\xvec^n]\big)\hatf_k\dd\xi\bigg)^2\bigg],
 \end{align*}
 where we used that $(x-y)^2\ge x^2/2-y^2$ for arbitrary $x,y\in\setR$.
 Now, on one hand,
 \begin{align*}
   \sum_{k=1}^{K-1}\bigg(\int_0^M V_x\big(\convX_\theh[\xvec^n]\big)\hatf_k\dd\xi\bigg)^2
   & \le \sup_{x\in I}\big(V_x(x)^2\big)\,\sum_{k=1}^{K-1}\bigg(\int_0^M \hatf_k\dd\xi\bigg)^2 \\
   & =\sup_{x\in I}\big(V_x(x)^2\big)\,\sum_{k=1}^{K-1}\frac{(\delta_k+\delta_{k+1})^2}4 
   \le M\delmax(\theh) \sup_{x\in I}\big(V_x(x)^2\big).
 \end{align*}
 And on the other hand,
 \begin{align*}
   \sum_{k=1}^{K-1}\big(\psi'(z^n_{k+1})-\psi'(z^n_k)\big)^2
   \ge 2\delmin(\theh)\,\sum_{k=1}^{K-1}\frac{\big(\psi'(z^n_{k+1})-\psi'(z^n_k)\big)^2}{\delta_k+\delta_{k+1}}.
 \end{align*}
 Combining these estimates, we obtain
 \begin{align*}
   \sum_{k=1}^{K-1}\frac{\big(\psi'(z^n_{k+1})-\psi'(z^n_k)\big)^2}{\delta_k+\delta_{k+1}}
   \le \frac{\delmax(\theh)}{\delmin(\theh)}
   \bigg(\frac12[\grd\nrj_\theh(\xvec_\Delta^n)]^T\Wmat^{-1}[\grd\nrj_\theh(\xvec_\Delta^n)]
   +  M \sup_{x\in I}\big(V_x(x)^2\big)\bigg).
 \end{align*}
 Multiply by $\tau$ and sum over $n=1$ to $n=N$.
 An application of \eqref{eq:ee} yields \eqref{eq:ei}.
\end{proof}

\section{Qualitative properties of the discretization}
\label{sct:quality}
Throughout this section, we fix a space-time discretization $\Delta=(\tau,\theh)$
and consider a given discrete solution $u_\Delta=(u_\Delta^n)_{n=0}^\infty$.

\subsection{Metric contraction}
\label{sct:contract}
One of the fundamental properties of our solution scheme is the preservation of the contraction property \eqref{eq:contract}.
\begin{prp}
 If $v_\Delta=(v_\Delta^n)_{n=0}^\infty$ is any other discrete solution,
 then 
 \begin{align}
   \label{eq:contract3}
   \wass\big(u_\Delta^n,v_\Delta^n)^2
   \le (1-2\Lambda\tau)^{-n}\wass\big(u_\Delta^0,v_\Delta^0\big)^2
 \end{align}
 for all $n\in\setN$.
\end{prp}
\begin{rmk}
 Since $(1-2\Lambda\tau)^n<\exp(-2\Lambda n\tau)$ for every $n\in\setN$,
 estimate \eqref{eq:contract3} is slightly worse for every $\tau>0$ than the limiting estimate \eqref{eq:contract}.
\end{rmk}
\begin{proof}
 For $\xvec_\Delta$, $\yvec_\Delta$ such that $u_\Delta^n=\convf_\theh[\xvec_\Delta^n]$, $v_\Delta^n=\convf_\theh[\yvec_\Delta^n]$
 we know by Proposition \ref{prp:existence} that
 \begin{align*}
   \Wmat(\xvec_\Delta^n-\xvec_\Delta^{n-1}) = -\tau\grd\nrj_\theh(\xvec_\Delta^n)
   \quad \text{and} \quad
   \Wmat(\yvec_\Delta^n-\yvec_\Delta^{n-1}) = -\tau\grd\nrj_\theh(\yvec_\Delta^n).
 \end{align*}
 Substracting these equations, we obtain
 \begin{align*}
   \Wmat^{1/2}(\xvec_\Delta^n-\yvec_\Delta^n) + \tau\Wmat^{-1/2}\big(\grd\nrj_\theh(\xvec_\Delta^n)-\grd\nrj_\theh(\yvec_\Delta^n)\big)
   =\Wmat^{1/2}(\xvec_\Delta^{n-1}-\yvec_\Delta^{n-1}),
 \end{align*}
 where $\Wmat^{1/2}$ and $\Wmat^{-1/2}$ are the (symmetric and positive definite) square roots of $\Wmat$ and its inverse, respectively;
 see Lemma \ref{lem:wasserstein}.
 Taking the norm on both sides yields
 \begin{align*}
   (\xvec_\Delta^n-\yvec_\Delta^n) \Wmat (\xvec_\Delta^n-\yvec_\Delta^n)
   & + 2\tau (\xvec_\Delta^n-\yvec_\Delta^n)^T \big(\grd\nrj_\theh(\xvec_\Delta^n)-\grd\nrj_\theh(\yvec_\Delta^n)\big) \\
   & \le (\xvec_\Delta^{n-1}-\yvec_\Delta^{n-1})^T \Wmat(\xvec_\Delta^{n-1}-\yvec_\Delta^{n-1}).
 \end{align*}
 Combining this with the convexity property \eqref{eq:convex3}, 
 we arrive at the recursive relation
 \begin{align*}
   (1-2\Lambda\tau) (\xvec_\Delta^n-\yvec_\Delta^n)^T\Wmat (\xvec_\Delta^n-\yvec_\Delta^n)
   \le (\xvec_\Delta^{n-1}-\yvec_\Delta^{n-1})^T \Wmat(\xvec_\Delta^{n-1}-\yvec_\Delta^{n-1}).   
 \end{align*}
 Iteration of this estimate and application of \eqref{eq:anothermiracle} yields \eqref{eq:contract3}.
\end{proof}

\subsection{The maximum and minimum principles}
\label{sct:minmax}
Recall that $\psi''$ is a positive and non-increasing function by Lemma \ref{lem:phipsi},
and recall the definition of $\alpha(\theh)$ from \eqref{eq:deltamin}.
\begin{prp}[Minimum principle]
 \label{prp:minprinc}
 Assume that $\Lambda\tau<1/2$,
 and assume further that $\theh$ and $\tau$ are related by the \emph{inverse CFL condition}
 \begin{align}
   \label{eq:flc}
   \delmax(\theh)^2 \le 6\psi''\big( Z^*_T\big) \tau,
   \quad \text{where} \quad \underline Z^*_T := \frac{6\alpha(\theh)e^{2T\Lambda}}{\min_x u_\Delta^0}.
 \end{align}
 Then, for every $n$ with $n\tau\le T$,
 \begin{align}
   \label{eq:minprinc}
   \min_x u_\Delta^n \ge e^{-2n\tau\Lambda}\min_xu_\Delta^0.
 \end{align}
\end{prp}
\begin{proof}
 The minimum principle \eqref{eq:minprinc} for $u_\Delta$ is equivalent to a maximum principle for $\gvec_\Delta$.
 By induction on $n$, we prove
 \begin{align}
   \label{eq:maxprinc}
   Z^{(n)}:=\max_k z^n_k \le (1-\Lambda\tau)^{-n}\max_k z^0_k,
 \end{align}
 which is a slightly sharper estimate than \eqref{eq:minprinc} 
 since $(1-\tau\Lambda)>e^{-2\tau\Lambda}$ under our assumption $\tau\Lambda<1/2$.
 So fix $n$ with $n\tau\le T$ and assume that $z^{n-1}_k\le Z^{(n-1)}$ for all $k$.
 Suppose that $Z^{(n)} = z^n_m$ for some $m\in\{2,\ldots,K-1\}$.
 We  estimate the integral term in the Euler-Lagrange equation \eqref{eq:el3},
 using the positivity of the $z^n_k$ and that $\hatF_m\ge0$:
 \begin{align*}
   [\Wnew\gvec^n]_m
   &\le [\Wnew\gvec^{n-1}]_m
   + \tau\big((1-\gamma_m)\psi'(z^n_{m+1}) - 2\psi'(z^n_m) + (1+\gamma_m)\psi'(z^n_{m-1})\big) \\
   &\qquad + \Lambda\tau\sum_{k=m-1}^{m+1}z^n_k\int_{\xi_{k-1}}^{\xi_k}\hatF_m\dd\xi.
 \end{align*}
 Recalling \eqref{eq:3}, we conclude further that 
 \begin{align}
   \label{eq:heavy}
   \begin{split}
   (1-\Lambda\tau)[\Wnew\gvec^n]_m
   &\le [\Wnew\gvec^{n-1}]_m \\
   & + (1-\gamma_m) \tau\big(\psi'(z^n_{m+1}) -\psi'(z^n_m)\big) 
   + (1+\gamma_m)\tau\big(\psi'(z^n_{m-1}) -\psi'(z^n_m)\big).
   \end{split}
 \end{align}
 Since $Z^{(n)}=z^n_m\ge z^n_k$ for all $k$, 
 and since $\psi''>0$ is non-increasing, it follows that
 \begin{align}
   \label{eq:light}
   \begin{split}
     &\psi'(z^n_{m+1}) -\psi'(z^n_m) \le -\psi''(Z^{(n)})(z^n_m-z^n_{m+1}), \\
     &\psi'(z^n_{m-1}) -\psi'(z^n_m) \le -\psi''(Z^{(n)})(z^n_m-z^n_{m-1}).      
   \end{split}
 \end{align}
 In particular, these terms are non-positive, and so \eqref{eq:heavy} implies
 \begin{align*}
   (1-\Lambda\tau)\Wnew_{m,m}Z^{(n)}\le\sigma_m Z^{(n-1)},
   \quad\text{with}\quad
   \sigma_m = \Wnew_{m,m-1}+\Wnew_{m,m}+\Wnew_{m,m+1}.
 \end{align*}
 From the explicit form of $\Wnew$ in \eqref{eq:wnew}, we obtain $\sigma_m\le3\alpha(\theh)\Wnew_{m,m}$, 
 and by means of the induction hypotheses, 
 we conclude the rough bound
 \begin{align}
   \label{eq:rough}
   Z^{(n)} \le 6\alpha(\theh)Z^{n-1} \le 6\alpha(\theh) (1-\Lambda\tau)^{-(n-1)}\max_k z^0_k \le Z^*_T.
 \end{align}
 We return to \eqref{eq:heavy}, insert \eqref{eq:light}, use that $\psi''$ is non-increasing,
 and find after some manipulations:
 \begin{align*}
   (1-\Lambda\tau)\sigma_m Z^{(n)} \le \sigma_m Z^{(n-1)} 
   &+ (1-\gamma_m) \Big(\frac{1-\Lambda\tau}6\delta_{m+1}^2-\tau\psi''(Z^*_T)\Big)(z^n_m-z^n_{m-1}) \\
   &+ (1+\gamma_m) \Big(\frac{1-\Lambda\tau}6\delta_{m-1}^2-\tau\psi''(Z^*_T)\Big)(z^n_m-z^n_{m+1}).
 \end{align*}
 The inverse CFL condition \eqref{eq:flc} implies non-positivity of the last two terms,
 so \eqref{eq:rough} refines to
 \begin{align*}
   Z^{(n)} \le (1-\Lambda\tau)^{-1}Z^{(n-1)} \le  (1-\Lambda\tau)^{-n}\max_k z^0_k.
 \end{align*}
 If the maximum is attained at one of the boundary points, $m=1$ or $m=K$, 
 then a similar calculation can be carried out using \eqref{eq:el2a} or \eqref{eq:el2b}, respectively, instead of \eqref{eq:el3}.
\end{proof}
Similarly to the minimum princple, one obtains the following maximum principle.
Notice that the inverse CFL condition is the same as before, i.e., it involves $\min_x u_\Delta^0$ and not $\max_x u_\Delta^0$.
\begin{prp}[Maximum principle]
 Let all the hypotheses of Proposition \ref{prp:minprinc} hold,
 and assume --- in addition to \eqref{eq:flc} --- that
 \begin{align*}
   (1+\lambda\tau)\delmax(\theh)^2\le6\tau\psi''(Z^*_T), \quad \text{where} \quad \lambda:=\max_I V_{xx}\ge0.
 \end{align*}
 Then, for every $n$ with $n\tau\le T$,
 \begin{align*}
   \max_x u_\Delta^n \le e^{2n\tau\lambda}\max_x u_\Delta^0.
 \end{align*}
\end{prp}
We omit the proof, which is very similar to the one given above.

\subsection{Regularity}
Introduce the function $\Phi:\setR_+\to\setR$ by
\begin{align}
 \label{eq:bigpsi}
 \Phi(r) = \int_0^r \sqrt{\rho}\phi''(\rho)\dd\rho.
\end{align}
$\Phi$ is strictly increasing, with
\begin{align}
 \label{eq:Phitopsi}
 \Phi'(r) = \sqrt{r}\phi''(r) = r^{-5/2}\psi''(r^{-1})
\end{align}
by definition of $\psi$ in \eqref{eq:psi}.
Further, $\Phi^2$ grows asymptotically faster than $\prss$,
\begin{align}
 \label{eq:superlinear}
 \lim_{r\to\infty}\frac {\Phi(r)^2}{\prss(r)} = + \infty,
\end{align}
which is easily verified by l'Hospitals rule:
\begin{align*}
 \lim_{r\to\infty}\frac {\Phi(r)^2}{\prss(r)} 
 = 2\lim_{r\to\infty}\frac{\Phi(r)\Phi'(r)}{\prss'(r)} 
 = 2\lim_{r\to\infty}\frac{\Phi(r)}{\sqrt{r}}
 = 4\lim_{r\to\infty}\big(\sqrt{r}\Phi'(r)\big)
 = 4\lim_{r\to\infty}\prss'(r)
 = +\infty,  
\end{align*}
due to \eqref{eq:pressume}.
In particular, $\Phi^2$ is superlinear at infinity. 
\begin{prp}
 For every $N\in\setN$,
 \begin{align}
   \label{eq:regularity2}
   \tau\sum_{n=1}^N \var{\Phi(u^n)^2} &\le \thec\big(\anrj_\theh(u_\Delta^0),\alpha(\theh),N\tau\big),
 \end{align}
 where
 \begin{align}
   \label{eq:thec}
   \thec(E,\alpha,T) = T\Phi\Big(\frac{M}{b-a}\Big)^2 + 6(b-a)\alpha(1+\alpha)\big[E-\underline\anrj+MT\sup_{x\in I}(V_x(x)^2)\big].
 \end{align}
\end{prp}
Estimate \eqref{eq:regularity2} mimicks the standard energy dissipation estimate for \eqref{eq:adde}, 
which is
\begin{align*}
 - \frac{\dd}{\dd t}\anrj(u) = \int_I u\big[\phi'(u)+V\big]_x^2\dd x \ge \int_I \Phi(u)_x^2\dd x - C.
\end{align*}
Roughly speaking, the integral of the spatial derivative is replaced by the total variation,
\begin{align*}
 \var{\Phi(u_\Delta^n)^2} = \sum_{k=1}^{K-1} \big|\Phi(u^n_{k+1})^2-\Phi(u^n_k)^2\big|,
\end{align*}
which is about the maximal regularity that can be associated to a piecewise constant function.
\begin{proof}
 Recall that $u^n_k=1/z^n_k=\delta_k/(x^n_k-x^n_{k-1})$.
 On one hand, using \eqref{eq:Phitopsi},
 \begin{align*}
   \big(\Phi(u^n_{k+1})-\Phi(u^n_k)\big)^2
   &= \bigg(\int_{u^n_k}^{u^n_{k+1}}\Phi'(r)\dd r\bigg)^2 \\
   &= \bigg(\int_{u^n_k}^{u^n_{k+1}}r^{-5/2}\psi''(r^{-1})\dd r\bigg)^2 
    = \bigg(\int_{1/u^n_{k+1}}^{1/u^n_k}\sqrt{s}\psi''(s)\dd s\bigg)^2 \\
    &\le \max(z^n_k,z^n_{k+1})\bigg(\int_{z^n_{k+1}}^{z^n_k}\psi''(s)\dd s\bigg)^2 
    = \max(z^n_k,z^n_{k+1})\big(\psi'(z^n_k)-\psi'(z^n_{k+1})\big)^2.
 \end{align*}
 On the other hand, recalling the definition of $\alpha(\theh)$ in \eqref{eq:deltamin},
 \begin{align*}
   (\delta_k+\delta_{k+1})(z^n_k+z^n_{k+1})\le \big(1+\alpha(\theh)\big)(z^n_k\delta_k+z^n_{k+1}\delta_{k+1}).
 \end{align*}
 We combine these estimates and sum over $k=1,\ldots,K-1$ to obtain
 \begin{align*}
   \sum_{k=1}^{K-1}\frac{\big(\Phi(u^n_{k+1})-\Phi(u^n_k)\big)^2}{z^n_k\delta_k+z^n_{k+1}\delta_{k+1}}
   \le S^n := \big(1+\alpha(\theh)\big)\sum_{k=1}^{K-1} \frac{\big(\psi'(z^n_k)-\psi'(z^n_{k+1})\big)^2}{\delta_{k+1}+\delta_k}.
 \end{align*}
 Since  $x^n_{k+1}-x^n_{k-1}=z^n_k\delta_k+z^n_{k+1}\delta_{k+1}$,
 it follows further that
 for arbitrary $k,\ell\in\{1,2,\ldots,K\}$ with $k\le\ell$:
 \begin{align}
   \label{eq:key}
   \begin{split}
     |\Phi(u^n_\ell)-\Phi(u^n_k)| 
     &\le \sum_{m=k}^{\ell-1} |\Phi(u^n_{m+1})-\Phi(u^n_m)| \\
     &\le \bigg(\sum_{m=k}^{\ell-1}(x^n_{m+1}-x^n_{m-1})\bigg)^{1/2}
     \bigg(\sum_{m=k}^{\ell-1} \frac{\big(\Phi(u^n_{m+1})-\Phi(u^n_m)\big)^2}{z^n_m\delta_m+z^n_{m+1}\delta_{m+1}}\bigg)^{1/2} \\
     &\le \big(2(b-a)S^n\big)^{1/2}. 
   \end{split}
 \end{align}
 Since $u$ has average value $M/(b-a)$, there exists an index $m^*$ with $u_{m^*}\le M/(b-a)$,
 which implies the absolute bound
 \begin{align*}
   \max_{1\le k\le K}\Phi(u_k) \le \Phi(u_{m^*}) + |\Phi(u_k)-\Phi(u_{m^*})| \le  \Phi\Big(\frac{M}{b-a}\Big) + \big(2(b-a)S^n\big)^{1/2}.
 \end{align*}
 So, finally,
 \begin{align*}
   \var{\Phi(u_\Delta^n)^2} 
   = \sum_{k=1}^{K-1}\big|\Phi(u^n_{k+1})^2-\Phi(u^n_k)^2\big|
   & \le 2\max_{1\le k\le K}\Phi(u^n_k) \sum_{k=1}^{K-1}|\Phi(u^n_{k+1})-\Phi(u^n_k)| \\
   &\le 2 \Big[\Phi\Big(\frac{M}{b-a}\Big) + \big(2(b-a)S^n\big)^{1/2}\Big]\big(2(b-a)S^n\big)^{1/2} \\
   &\le \Phi\Big(\frac{M}{b-a}\Big)^2 + 6(b-a)S.
 \end{align*}
 To obtain \eqref{eq:regularity2}, sum with respect to $n=1,\ldots,N$
 and use the energy estimate \eqref{eq:ei}.
\end{proof}

\section{Convergence}
\label{sct:convergence}
In this section, we prove Theorem \ref{thm:main}.
Let a time horizont $T>0$ and an initial condition $u^0\in L^1(I)$ with $\anrj(u^0)<\infty$ be given.
We consider a family $\Delta_j=(\tau_j,\theh_j)$ of time-space discretizations, with $j\in\setN$. 
Accordingly, we denote by $K_j$ the number of nodes of $\theh_j$, and $N_j$ is the smallest integer with $\tau_jN_j\ge T$.

Throughout this section, we assume all the hypotheses of Theorem \ref{thm:main}:
\begin{itemize}
\item $\tau_j\downarrow0$ and $\delmax(\theh_j)\downarrow0$ as $j\to\infty$;
\item initial conditions $u_{\Delta_j}^0\in\densNj$ are given for each $j$, 
 such that $u_{\Delta_j}^0$ converges to $u^0$ weakly in $L^1(I)$.
\item uniformly in $j\in\setN$,
 \begin{align}
   \label{eq:uniform}
   \alpha(\theh_j)\le\overline\alpha< \infty,
   \quad
   \anrj(u_{\Delta_j}^0) \le \overline\anrj < \infty,
   \quad
   \delmax(\theh_j)^2 \le 6\psi''\bigg(\frac{6\overline\alpha e^{2\Lambda T}}{\min_x u_{\Delta_j}^0}\bigg)\tau_j.
 \end{align}
\end{itemize}
Denote by $(u_{\Delta_j}^n)_{n=0}^\infty$ the corresponding discrete solutions obtained as in Proposition \ref{prp:existence},
and introduce the time-interpolated functions $\bar u_{\Delta_j}:[0,T]\to\densNj$ 
by
\begin{align}
 \label{eq:tinterpolate}
 \bar u_{\Delta_j}(t;x) = u_{\Delta_j}^n(x) \quad \text{for all $t\in\big((n-1)\tau_j,n\tau_j\big]\cap[0,T]$}.
\end{align}
The following preliminary result plays an important role in the convergence proof.
\begin{lem}
 With the maximal mesh width  $\delmax(\xvec)$ defined in \eqref{eq:delmaxx},
 we have
 \begin{align}
   \label{eq:xuniform}
   \max_{n\le N_j}\delmax(\xvec_{\Delta_j}^n) \to 0 \quad \text{as $j\to\infty$}.
 \end{align}
\end{lem}
\begin{proof}
 This is a consequence of the minimum principle and hypothesis \eqref{eq:uniform}.
 It follows from the assumption $\lim_{r\downarrow0}\prss'(r)<\infty$ in \eqref{eq:pressume} that
 \begin{align*}
   \psi''\bigg(\frac{e^{2\Lambda T}}{\min_x u_{\Delta_j}^0}\bigg) 
   \le C \bigg(\frac{\min_x u_{\Delta_j}^0}{e^{2\Lambda T}}\bigg)^2
 \end{align*}
 for some appropriate constant $C$.
 Recalling that $(x_k-x_{k-1})u_k=\delta_k$,
 Proposition \ref{prp:minprinc} thus implies --- uniformly in $j$ and $n\le N_j$ --- that
 \begin{align*}
   \delmax(\xvec_{\Delta_j}^n)^2 
   \le \bigg(\frac{\delmax(\theh_j)}{\min_x u_{\Delta_j}^n} \bigg)^2
   \le \delmax(\theh_j)^2\bigg(\frac{e^{2\Lambda T}}{\min_x u_{\Delta_j}^0}\bigg)^2
   \le \frac{C \delmax(\theh_j)^2}{\psi''\big(\frac{e^{2\Lambda T}}{\min_x u_{\Delta_j}^0}\big)}
   \le 6C\tau_j,
 \end{align*}
 using hypothesis \eqref{eq:uniform}.
 Since $\tau_j\downarrow0$ as $j\to\infty$, this proves the claim.
\end{proof}

\subsection{Compactness in $\wass$}
The following weak convergence result is a well-known consequence of the energy estimate \eqref{eq:we} in combination with the Arzel\`{a}-Ascoli theorem.
\begin{prp}
 \label{prp:aa}
 More precisely, every subsequence of $(\bar u_{\Delta_j})_{j\in\setN}$ contains a sub-subsequence
 that converges uniformly w.r.t.\ $t\in[0,T]$ in $\wass$ to a limit curve $u_*\in C^{1/2}([0,T];\wass)$.
\end{prp}
A proof can be obtained by application of \cite[Proposition 3.3.1]{AGS}.

\subsection{Compactness in $L^1$}
The following compactness property on the $\bar u_{\Delta_j}$ is at the basis for our convergence proof.
\begin{prp}
 \label{prp:compact}
 Every subsequence of $(\bar u_{\Delta_j})_{j\in\setN}$ contains a sub-subsequence such that
 the respective $\bar u_{\Delta_j}$ converge to some $u_*$,
 and the $\prss(\bar u_{\Delta_j})$ converge to $\prss(u_*)$,
 both strongly in $L^1([0,T]\times I)$.
\end{prp}
The proof of this proposition is an application of the Aubin-Lions compactness principle. 
Specifically, we use:
\begin{thm}\label{thm:compactness}[Adapted from Theorem 2 in \cite{RosS}]
 Assume that: 
 \begin{enumerate}
 \item There is a \emph{normal coercive integrand} $\nci:L^1(I)\to[0,\infty]$, i.e.,
   $\nci$ is measurable, lower semi-continuous and has compact sublevels in $L^1(I)$,
   for which the following is true:
   \begin{align}
     \label{eq:savare1}
     \sup_{j\in\setN}\int_0^T\nci\big(\bar u_{\Delta_j}(t)\big)\dd t < \infty.
   \end{align}
 \item The $\bar u_{\Delta_j}$ are integral equicontinuous with respect to $\wass$,
   \begin{align}
     \label{eq:savare2}
     \lim_{h\downarrow0}\sup_{j\in\setN}\int_0^{T-h} \wass\big(\bar u_{\Delta_j}(t+h),\bar u_{\Delta_j}(t)\big)\dd t = 0.
   \end{align}
 \end{enumerate}
 Then the sequence $(\bar u_{\Delta_j})_{j\in\setN}$ is relatively compact in $L^1([0,T]\times I)$.
\end{thm}
In order to define $\nci$, we first recall that the \emph{total variation} of a function $f\in L^1(I)$ is given by
\begin{align*}
 \tv{f} := \sup \bigg\{ \int_a^b f(x)\varphi'(x)\dd x \,\bigg|\, \varphi\in C^1_0(I),\,\sup_{x\in I} |\varphi(x)|\le 1\bigg\}.
\end{align*}
It is easily checked that for piecewise constant densities $u\in\densNj$, and with $\Phi$ from \eqref{eq:bigpsi}, 
we have
\begin{align}
 \label{eq:tvvar}
 \tv{\Phi(u)^2} = \var{\Phi(u)^2} = \sum_{k=1}^{K_j-1}\big|\Phi(u_{k+1})^2-\Phi(u_k)^2\big|.
\end{align}
Now let $\nci:L^1(I)\to\setR\cup\{+\infty\}$ be given by
\begin{align*}
 \nci(u) =
 \begin{cases}
   \tv{\Phi(u)^2} & \text{if $u\in\overline{\dens}$}, \\
   +\infty & \text{otherwise};
 \end{cases}
\end{align*}
where $\overline{\dens}$ denotes the closure of $\dens$ in $L^1(I)$,
which consists of all non-negative $L^1$-functions with integral equal to $M$.
\begin{lem}
 \label{lem:nci}
 The functional $\nci$ defined above is lower semi-continuous and has relatively compact sublevels.
\end{lem}
\begin{proof}[Proof of Lemma \ref{lem:nci}]
 Let $A_c:=\nci^{-1}((-\infty;c])\subset L^1(I)$ be a sublevel of $\nci$.
 By \cite[Theorem 1.19]{Giusti}, the set $B_c:=\{\Phi(u)^2\,|\,u\in A_c\}$ is relatively compact in $L^1(I)$;
 here we use that our domain $I$ is an interval, 
 so that $\tv{\Phi(u)^2}\le c$ and $\int_I u(x)\dd x=M$ induce a uniform bound on the BV-norm of $\Phi(u)^2$.

 Thus, if $(u_\ell)$ is a sequence in $A_c$, converging to $u_0$ in $L^1(I)$,
 then also $(\Phi(u_\ell)^2)$ converges to $\Phi(u_0)^2$ in $L^1(I)$.
 By lower semi-continuity of the total variation $\tv{\cdot}$ \cite[Theorem 1.9]{Giusti},
 the lower semi-continuity of $\nci$ follows.

 To conclude compactness of $A_c$, 
 it suffices to prove that the mapping $u\mapsto\Phi(u)^2$ is $L^1(I)$-continuously invertible.
 For that, let a sequence $(f_\ell)_{\ell\in\setN}$ in $B_c$ be given, which converges to some $f_0$ in $L^1(I)$.
 Since the map $r\mapsto\Phi(r)^2$ is strictly increasing, positive, and continuous with superlinear growth \eqref{eq:superlinear}, 
 it possesses a strictly increasing, positive and continuous inverse with sublinear growth.
 Hence, there are a uniquely determined sequence of functions $u_\ell\in A_c$ such that $\Phi(u_\ell)^2=f_\ell$ for all $\ell\in\setN$,
 and a unique $u_0\in A_c$ with $\Phi(u_0)^2=f_0$.
 We wish to show that $u_\ell$ converges to $u_0$ in $L^1(I)$.
 By standard arguments, we can assume without loss of generality that the $f_\ell$ converge to $f_0$ pointwise a.e.
 By continuous invertibility of $r\mapsto\Phi(r)^2$, the $u_j$ converge to $u_0$ pointwise a.e.
 Moreover, by construction,
 \begin{align*}
   \sup_{\ell\in\setN} \int_a^b \Phi(u_\ell(x))^2\dd x = \sup_{\ell\in\setN}\int_a^b f_\ell(x)\dd x < \infty,
 \end{align*}
 so we can invoke Vitali's theorem --- recall the superlinear growth \eqref{eq:superlinear} --- to conclude strong convergence of $u_\ell$ to $u_0$.
\end{proof}
\begin{proof}[Proof of Proposition \ref{prp:compact}]
 It suffices to show that every subsequence of $(\bar u_{\Delta_j})_{j\in\setN}$ contains a sub-subsequence which is relatively compact.
 In view of Proposition \ref{prp:aa}, we may thus assume --- without loss of generality --- 
 that $(\bar u_{\Delta_j})_{j\in\setN}$ converges uniformly w.r.t.\ $t\in[0,T]$ in $\wass$ to a curve $u_*\in C^{1/2}([0,T];\wass)$.
 The verification of \eqref{eq:savare2} then becomes an easy exercise, which is left to the reader.

 We verify \eqref{eq:savare1}.
 As remarked in \eqref{eq:tvvar}, we have
 \begin{align*}
   \nci(\bar u_{\Delta_j}^n) = \var{\Phi(u_{\Delta_j}^n)^2}
 \end{align*}
 for all $n=1,\ldots,N_j$.
 Thus, the regularity estimate \eqref{eq:regularity2} implies
 \begin{align*}
   \int_0^T\nci(\bar u_{\Delta_j})
   \le \tau_j\sum_{n=1}^{N_j} \nci(u_{\Delta_j}^n) 
   \le \thec\big(\overline\anrj,\overline\alpha,T+\tau_j\big),
 \end{align*}
 with $\thec$ given in \eqref{eq:thec},
 and with $\overline\anrj$, $\overline\alpha$ from \eqref{eq:uniform}.
 This yields the uniform bound \eqref{eq:savare1}.

 Thus Theorem \ref{thm:compactness} applies and provides relative compactness of $(\bar u_{\Delta_j})_{j\in\setN}$ in $L^1([0,T]\times I)$.
 Since $L^1$-convergence implies weak convergence, it actually follows that $\bar u_{\Delta_j}$ converges to $u_*$ in $L^1([0,T]\times I)$.
 Without loss of generality, we may even assume that $\bar u_{\Delta_j}$ converges to $u_*$ a.e.\ on $[0,T]\times I$.
 By continuity of $\prss$, also $\prss(\bar u_{\Delta_j})$ converges to $\prss(u_*)$ a.e.\ on $[0,T]\times I$.
 Further,
 \begin{align*}
   \int_0^T\int_I\Phi\big(\bar u_{\Delta_j}(t;x)\big)^2\dd x\dd t
   &\le (b-a)\tau\sum_{n=1}^{N_j}\Big[\Phi\Big(\frac{M}{b-a}\Big)^2 + \var{\Phi(\bar u_{\Delta_j}^n)^2}\Big] \\
   &\le 2(b-a)\thec\big(\overline\anrj,\overline\alpha,T+\tau_j\big),
 \end{align*}
 which is $j$-uniformly bounded because of the regularity estimate \eqref{eq:regularity2}.
 By the growth property \eqref{eq:superlinear} of $\Phi$, we can invoke Vitali's theorem to conclude 
 that $\prss(\bar u_{\Delta_j})$ tends to $\prss(u_*)$ in $L^1([0,T]\times I)$.
\end{proof}

\subsection{Weak formulation}
Combining the compactness results from Proposition \ref{prp:aa} and Proposition \ref{prp:compact},
we know that every subsequence of $(\Delta_j)_{j\in\setN}$ contains a sub-subsequence 
for which $\bar u_{\Delta_j}$ converges to some limit 
\begin{align*}
 u_*\in C^{1/2}([0,T];\wass)\cap L^1([0,T]\times I),
\end{align*}
uniformly w.r.t.\ $t\in[0,T]$ in $\wass$, and strongly in $L^1([0,T]\times I)$;
finally, also $\prss(\bar u_{\Delta_j})$ converges to $\prss(u_*)$ in $L^1([0,T]\times I)$.
To simplify notations, we denote that sub-subsequence simply by $(\bar u_{\Delta})$, bearing in mind that $\Delta=\Delta_j$.
In this subsection, we prove that every such limit $u_*$ is a weak solution to the initial value problem \eqref{eq:adde}.
\begin{prp}
 \label{prp:weak}
 $u_*$ satisfies the weak formulation
 \begin{align}
   \label{eq:weak}
   \int_0^T\int_I u_*\partial_t\varphi\dd x\dd t
   = \int_0^T\int_I u_*V_x\partial_x\varphi\dd x\dd t
   - \int_0^T\int_I \prss(u_*)\partial_{xx}\varphi\dd x\dd t
 \end{align}
 for all test functions $\varphi$ from
 \begin{align}
   \label{eq:testfunc}
   \tst:=\big\{\varphi\in C^\infty([0,T]\times I)\,\big|\,\supp\varphi\subset(0,T)\times I,\,\varphi_x(t;a)=\varphi_x(t;b)=0 \text{ f.a. $t\in[0,T]$}\big\}.
 \end{align}
 Moreover, $u_*$ attains the initial datum $u^0$ weakly-$\star$ as $t\downarrow0$.
\end{prp}
\begin{rmk}
 Since weak solutions to \eqref{eq:adde} are unique (this follows, e.g., by metric contraction of the gradient flow), 
 we conclude a posteriori that the entire sequence $(u_{\Delta_j})_{j\in\setN}$ converges to the solution $u_*$.  
\end{rmk}
\begin{rmk}
 By our definition \eqref{eq:testfunc} of test functions, 
 the weak formulation \eqref{eq:weak} automatically induces homogeneous Neumann boundary conditions on $u$.
\end{rmk}
The weak formulation \eqref{eq:weak} is obtained in the limit $j\to\infty$ from a certain fully discrete variant of the weak formulation.
The latter is derived by studying suitable variations of the minimizers $u_{\Delta_j}^n$.
In order to discuss this perturbation,
let $\rho\in C^\infty(I)$ with $\rho_x(a)=\rho_x(b)=0$ be given, 
and let $\kappa>0$ be such that
\begin{align}
 \label{eq:kappa}
 |\rho_x(x)|\le\kappa,\ |\rho_{xx}(x)|\le\kappa,\ |\rho_{xxx}(x)|\le\kappa \quad \text{for all $x\in I$}.
\end{align}
Further, fix $j\in\setN$ and also $n\in\setN$.
We shall omit the index $j$ in the following.

Introduce the functionals $\auxil:\xspc\to\setR$ and $\auxil_\theh:\xseq\to\setR$ by
\begin{align*}
 \auxil(\theX) = \int_0^M \rho(\theX)\dd\xi \quad\text{and}\quad \auxil_\theh(\xvec)=\auxil(\convX_\theh[\xvec]).
\end{align*}
The derivative of $\auxil_\theh$ is given by
\begin{align*}
 \big[\grd\auxil_\theh(\xvec)\big]_k = \int_0^M \rho_x\big(\convX_{\theh_j}[\xvec]\big)\hatf_k\dd\xi 
 \quad\text{for $k=1,\ldots,K-1$}.
\end{align*}
The variations we study are those induced by the \emph{gradient vector field} $\velo:=\Wmat^{-1}\grd\auxil_\theh$,
i.e.,
\begin{align}
 \label{eq:velo}
 \Wmat\velo(\xvec) = \grd\auxil_\theh(\xvec).
\end{align}
This vector field has a nice asymptotic expansion.
\begin{lem}
 \label{lem:velo}
 For every $\xvec\in\xseq$, define $\vec{\rho_x}(\xvec)\in\setR^{K-1}$ by
 \begin{align}
   \label{eq:rhox}
   \vec{\rho_x}(\xvec) &= \big(\rho_x(x_1),\rho_x(x_2),\ldots,\rho_x(x_{K-1})\big).
 \end{align}
 Then the residual vector $\nu=\velo(\xvec)-\vec{\rho_x}(\xvec)$ satisfies
 \begin{align}
   \label{eq:important}
   \nu^T\Wmat\nu \le C\kappa^2\alpha(\theh)M\delmax(\xvec)^2, 
 \end{align}
 with $\kappa$ defined in \eqref{eq:kappa} and some universal constant $C$.
\end{lem}
\begin{proof}
 We start by estimating
 \begin{align*}
   \mu := \Wmat\nu = \Wmat\big(\velo(\xvec)-\vec{\rho_x}[\xvec]\big) =\grd\auxil_\theh(\xvec)-\Wmat\vec{\rho_x}[\xvec].
 \end{align*}
 By definition, we have
 \begin{align*}
   \convX_\theh[\xvec](\xi) 
   = x_k + \frac{x_k-x_{k-1}}{\delta_k}(\xi-\xi_k)
   = x_{k-1} + \frac{x_k-x_{k-1}}{\delta_k}(\xi-\xi_{k-1})
 \end{align*}
 for all $\xi\in[\xi_{k-1},\xi_k]$.
 Using the explicit form of $\Wmat$ given in \eqref{eq:Wmat},
 we calculate:
 \begin{align*}
   \mu_k 
   &= \int_{\xi_{k-1}}^{\xi_k} \Big[\rho_x(\convX_\theh[\xvec])-\big(\frac23\rho_x(x_k)+\frac13\rho_x(x_{k-1})\big)\Big]\hatf_k\dd\xi \\
   & \qquad + \int_{\xi_{k}}^{\xi_{k+1}} \Big[\rho_x(\convX_\theh[\xvec])-\big(\frac23\rho_x(x_k)+\frac13\rho_x(x_{k+1})\big)\Big]\hatf_k\dd\xi \\
   &= (x_k-x_{k-1})\,\delta_k\int_0^1\Big[\frac{2\rho_{xx}(\hat x)}3(1-s)+\frac{\rho_{xx}(\check x)}3s\Big](1-s)\dd s \\
   & \qquad + (x_{k+1}-x_{k})\,\delta_{k+1}\int_0^1\Big[\frac{2\rho_{xx}(\hat x)}3(1-s)+\frac{\rho_{xx}(\check x)}3s\Big](1-s)\dd s ,
 \end{align*}
 where $\hat x,\check x\in I$ denote suitable intermediate values, depending on $s$.
 It follows that there is a universal constant $C_1$ such that
 \begin{align*}
   |\mu_k| \le C_1 \sup_{x\in I}|\rho_{xx}(x)|\max_{m=1,\ldots,K}(\delta_k+\delta_{k+1})\delmax(\xvec)^2
 \end{align*}
 for every $k=1,\ldots,K-1$.
 Recalling the lower estimate on $\Wmat$ in \eqref{eq:posdef},
 it follows for $\nu=\Wmat^{-1}\mu$ that
 \begin{align*}
   \nu^T\Wmat\nu 
   = \mu^T\Wmat^{-1}\mu 
   \le \frac6{\delmin(\theh)}\sum_{k=1}^{K-1} \mu_k^2 
   \le \bigg(\frac{6C_1^2\kappa^2}{\delmin(\theh)}\sum_{k=1}^{K-1}(\delta_{k+1}+\delta_k)^2\bigg)\delmax(\xvec)^2 
   \le 24C_1^2\kappa^2\alpha(\theh)M\delmax(\xvec)^2, 
 \end{align*}
 proving our claim \eqref{eq:important}.
\end{proof}
\begin{lem}
 With $\kappa$ satisfying \eqref{eq:kappa},
 the following estimate holds:
 \begin{align}
   \label{eq:fi}
   \frac1\tau\big(\auxil_\theh(\xvec_\Delta^n)-\auxil_\theh(\xvec_\Delta^{n-1})\big)
   &\le - \grd\nrj_\theh(\xvec_\Delta^n)^T\vec{\rho_x}(\xvec_\Delta^n) \\
   \label{eq:firem1}
   &\quad + \frac\kappa{2\tau}(\xvec_\Delta^n-\xvec_\Delta^{n-1})^T\Wmat(\xvec_\Delta^n-\xvec_\Delta^{n-1}) \\
   \label{eq:firem2}
   &\quad + C(M\alpha(\theh))^{1/2}\kappa\big(\grd\nrj_\theh(\xvec_\Delta^n)^T\Wmat^{-1}\grd\nrj_\theh(\xvec_\Delta^n)\big)^{1/2}\delmax(\xvec_\Delta^n).
 \end{align}
\end{lem}
\begin{proof}
 A Taylor expansion of $\rho$ yields for arbitrary $\theX,\theX'\in\xspc$:
 \begin{align*}
   \auxil(\theX')\ge\auxil(\theX) + \int_0^M \rho_x(\theX)\cdot(\theX'-\theX)\dd\xi - \frac\kappa2\int_0^M (\theX'-\theX)^2\dd\xi.
 \end{align*}
 With $\theX=\convX_\theh[\xvec_\Delta^n]$ and $\theX'=\convX_\theh[\xvec_\Delta^{n-1}]$, 
 we obtain
 \begin{align*}
   \frac1\tau\big(\auxil_\theh(\xvec_\Delta^n)-\auxil_\theh(\xvec_\Delta^{n-1})\big)
   & \le -\sum_{m=1}^{K-1}\bigg(\frac1\tau\big[\xvec_\Delta^n-\xvec_\Delta^{n-1}\big]_m\int_0^M\rho_x\big(\convX_\theh[\xvec_\Delta^n]\big)\hatf_m\dd\xi\bigg)\\
   & \qquad + \frac\kappa{2\tau}(\xvec_\Delta^n-\xvec_\Delta^{n-1})^T\Wmat(\xvec_\Delta^n-\xvec_\Delta^{n-1}).
 \end{align*}
 Using the definition \eqref{eq:velo} of $\velo$, the symmetry $\Wmat^T=\Wmat$, 
 and the Euler-Lagrange equations \eqref{eq:el}, 
 the sum can be rewritten as
 \begin{align*}
   \sum_{m=1}^{K-1} \frac1\tau\big[\xvec_\Delta^n-\xvec_\Delta^{n-1}\big]_m\big[\grd\auxil_\theh(\xvec_\Delta^n)\big]_m
   = \frac1\tau(\xvec_\Delta^n-\xvec_\Delta^{n-1})^T\Wmat\velo(\xvec_\Delta^n)
   = \grd\nrj_\theh(\xvec_\Delta^n)^T\velo(\xvec_\Delta^n).
 \end{align*}
 With the notations introduced in Lemma \ref{lem:velo}, we obtain further
 \begin{align*}
   \grd\nrj_\theh(\xvec_\Delta^n)^T\velo(\xvec_\Delta^n)
   & = \grd\nrj_\theh(\xvec_\Delta^n)^T\vec{\rho_x}(\xvec_\Delta^n) + \grd\nrj_\theh(\xvec_\Delta^n)^T\nu_\Delta^n \\
   & \le \grd\nrj_\theh(\xvec_\Delta^n)^T\vec{\rho_x}(\xvec_\Delta^n) 
   + \big(\grd\nrj_\theh(\xvec_\Delta^n)^T\Wmat^{-1}\grd\nrj_\theh(\xvec_\Delta^n)\big)^{1/2}\big((\nu_\Delta^n)^T\Wmat\nu_\Delta^n\big)^{1/2},  
 \end{align*}
 where the Cauchy-Schwarz inequality has been applied in the last step.
 The claim \eqref{eq:fi} now follows directly from the estimate \eqref{eq:important}.
\end{proof}
\begin{lem}
 For every $\xvec\in\xseq_\theh$,
 \begin{align}
   \label{eq:auxilrep}
   \auxil_\theh(\xvec) &= \int_I \rho(x)\convf_\theh[\xvec](x)\dd x, \\
   \label{eq:weakform}
   \grd\nrj_\theh(\xvec)^T\vec{\rho_x}(\xvec)
   &= - \int_I \prss\big(\convf_\theh[\xvec](x)\big)\rho_{xx}(\hat x)\dd x + \int_I V_x(x)\rho_x(\check x)\convf_\theh[\xvec](x)\dd x,
 \end{align}
 where $\hat x,\check x\in I$ are $x$-dependent quantities satisfying
 \begin{align*}
   |\hat x-x|,\,|\check x-x|<\delmax(\xvec).
 \end{align*}
\end{lem}
\begin{proof}
 Relation \eqref{eq:auxilrep} is a direct consequence of
 \begin{align*}
   \int_{\xi_{k-1}}^{\xi_k} \rho\big(x_k\hatf_k+x_{k-1}\hatf_{k-1}\big)\dd\xi
   = \int_{x_{k-1}}^{x_k} \rho(x)\frac{\delta_k}{x_k-x_{k-1}}\dd x,
 \end{align*}
 which follows by a change of variables $x=x_k\hatf_k(\xi)+x_{k-1}\hatf_{k-1}(\xi)$.
 To prove \eqref{eq:weakform}, first observe that \eqref{eq:gradient} implies
 \begin{align*}
   \grd\nrj_\theh(\xvec)^T\vec{\rho_x}(\xvec)
   & = -\sum_{k=1}^{K-1} \bigg[\psi'\Big(\frac{x_{k+1}-x_k}{\delta_{k+1}}\Big) - \psi'\Big(\frac{x_k-x_{k-1}}{\delta_k}\Big)\bigg]\rho_x(x_k) \\
   & \qquad + \int_0^M V_x\big(\convX_\theh[\xvec]\big)\sum_{k=1}^{K-1}\rho_x(x_k)\hatf_k\dd\xi. 
 \end{align*}
 We consider both terms on the right hand side separately.
 For the first, we obtain, using that $\psi'(1/r)=-P(r)$ and that $\rho_x(x_0)=\rho_x(x_K)$,
 \begin{align*}
   -\sum_{k=1}^{K-1} \bigg[\psi'\Big(\frac{x_{k+1}-x_k}{\delta_{k+1}}\Big) - \psi'\Big(\frac{x_k-x_{k-1}}{\delta_k}\Big)\bigg]\rho_x(x_k)
   &= \sum_{k=1}^K \psi'\Big(\frac{x_k-x_{k-1}}{\delta_k}\Big)\big(\rho_x(x_k)-\rho_x(x_{k-1})\big) \\
   & = - \sum_{k=1}^K \int_{x_{k-1}}^{x_k} \prss\Big(\frac {\delta_k}{x_k-x_{k-1}}\Big)\rho_{xx}(\hat x_k)\dd x,
 \end{align*}
 with a suitable $\hat x_k\in(x_{k-1},x_k)$ by the intermediate value theorem.
 For the other term, we perform a change of variables:
 \begin{align*}
   \int_{\xi_{k-1}}^{\xi_k} V_x\big(x_k\hatf_k+x_{k-1}\hatf_{k-1}\big)\big(\rho_x(x_k)\hatf_k+\rho_x(x_{k-1})\hatf_{k-1}\big)\dd\xi
   = \int_{x_{k-1}}^{x_k} V_x(x)\rho_x(\check x)\frac{\delta_k}{x_k-x_{k-1}}\dd x,
 \end{align*}
 with some $x$-dependent intermediate value $\check x\in(x_{k-1},x_k)$.
 Summation over $k=1,\ldots,K$ provides \eqref{eq:weakform}.
\end{proof}
\begin{lem}
 Let $\vartheta\in C^\infty_c(0,T)$ be a non-negative test function of compact support in $(0,T)$.
 Then 
 \begin{align}
   \label{eq:preweak}
   \int_0^T \int_I \vartheta'(t) \rho(x)u_*(t;x)\dd x \dd t
   \le \int_0^T \int_I \vartheta(t) \big[ - P(u_*)\rho_{xx} + V_x\rho_xu\big]\dd x\dd t.
 \end{align}
\end{lem}
\begin{proof}
 Multiply inequality \eqref{eq:fi}--\eqref{eq:firem2} by $\tau\vartheta(n\tau)\ge0$, and sum over $n=1,\ldots,N_\tau$.
 On the left-hand side, it follows by means of \eqref{eq:auxilrep} that
 \begin{align*}
   \tau\sum_{n=1}^{N_\tau} \vartheta(n\tau)\frac{\auxil_\theh(\xvec_\Delta^n)-\auxil_\theh(\xvec_\Delta^{n-1})}\tau
   &= - \tau \sum_{n=0}^{N_\tau-1} \frac{\vartheta((n+1)\tau)-\vartheta(n\tau)}\tau \auxil_\theh(\xvec_\Delta^n) \\
   &= - \int_0^T \int_I \widehat{\vartheta'_\tau}(t) \rho(x)\bar u_\Delta(t;x)\dd x\dd t,
 \end{align*}
 where the sequence of piecewise constant functions $\widehat{\vartheta'_\tau}$ converge to $\vartheta'$ uniformly on $[0,T]$.
 The strong convergence of $\bar u_\Delta$ to $u_*$ in $L^1$ is sufficient to pass to the limit $j\to\infty$.

 On the right-hand side, the first term can be rewritten using \eqref{eq:weakform}:
 \begin{align*}
   \tau\sum_{n=1}^{N_\tau} \vartheta(n\tau)\big(\grd\nrj_\theh(\xvec_\Delta^n)^T\vec{\rho_x}(\xvec_\Delta^n)\big)
   = \int_0^T \bar\vartheta_\tau(t) \int_I \big[-P(\bar u_\Delta)\rho_{xx}(\hat x) + V_x(x)\rho_x(\check x)\bar u_\Delta\big] \dd x\dd t,
 \end{align*}
 where, for $(n-1)\tau<t\le n\tau$, we know that $|x-\hat x|,|x-\check x|<\delmax(\xvec_\Delta^n)$.
 The convergence \eqref{eq:xuniform} implies that $\rho_{xx}(\hat x)$ and $\rho_x(\check x)$
 converge to their respective limits $\rho_{xx}(x)$ and $\rho_x(x)$ uniformly in $x\in I$.
 Likewise, the piecewise constant interpolants $\bar\vartheta_\tau$ converge to $\vartheta$ uniformly on $[0,T]$.
 The strong convergence of $\bar u_\Delta$ and of $P(\bar u_\Delta)$ to their respective limits $u_*$ and $P(u_*)$ in $L^1([0,T]\times I)$
 thus suffices to pass to the limit with the integral.

 We still need to estimate the remainder terms, resulting from \eqref{eq:firem1}\&\eqref{eq:firem2}.
 Observe that
 \begin{align*}
   \tau\sum_{n=1}^{N_\tau} \frac1{2\tau}(\xvec_\Delta^n-\xvec_\Delta^{n-1})^T\Wmat (\xvec_\Delta^n-\xvec_\Delta^{n-1})
   = \tau\,\frac1{2\tau} \sum_{n=1}^N \wass\big(\convf_\theh[\xvec_\Delta^n],\convf_\theh[\xvec_\Delta^{n-1}]\big)^2
   \le \tau\big(\overline\anrj-\underline\anrj)
 \end{align*}
 by \eqref{eq:we} and condition \eqref{eq:uniform}.
 Recalling that $\tau$ depends on $j$, with $\tau_j\downarrow 0$, this expression vanishes in the limit.
 Similarly, by \eqref{eq:ee} and \eqref{eq:uniform},
 \begin{align*}
   \tau\sum_{n=1}^{N_\tau}\big(\grd\nrj_\theh(\xvec_\Delta^n)^T\Wmat^{-1}\grd\nrj_\theh(\xvec_\Delta^n)\big)^{1/2}
   &\le (2T)^{1/2}\bigg(\frac\tau2\sum_{n=1}^{N_\tau}\grd\nrj_\theh(\xvec_\Delta^n)^T\Wmat^{-1}\grd\nrj_\theh(\xvec_\Delta^n)\bigg)^{1/2} \\
   &\le (2T)^{1/2}\big(\overline\anrj-\underline\anrj)^{1/2},
 \end{align*}
 which is $j$-uniformly bounded.
 Hence, the remainder term resulting from \eqref{eq:firem2} vanishes in the limit $j\to\infty$
 because of \eqref{eq:xuniform} and of \eqref{eq:uniform}.
\end{proof}
\begin{proof}[Proof of Proposition \ref{prp:weak}]
 First, observe that inequality \eqref{eq:preweak} is actually an equality, since $\rho$ can be replaced by $-\rho$ everywhere.
 Next, recall that functions $\testf$ of the type $\testf(t,x)=\vartheta(t)\rho(x)$, 
 where $\vartheta\in C^\infty_c(0,T)$ is non-negative, and $\rho\in C^\infty(I)$ satisfies $\rho_x(a)=\rho_x(b)=0$,
 are dense in the set $\tst$ of test functions.
 Thus, the weak formulation \eqref{eq:weak} holds for all $\testf\in\tst$.

 It remains to prove that the solution attains the initial datum $u^0$ weakly-$\star$ as $t\downarrow0$. 
 However, this is a trivial consequence of the H\"older regularity of $u_*\in C^{1/2}([0,T];\wass)$, 
 of the uniform convergence of $(\bar u_{\Delta_j})_{j\in\setN}$ to $u_*$ w.r.t.\ $t\in[0,T]$ in $\wass$,
 and of the approximation of $u^0$ by $u_{\Delta_j}^0$.
\end{proof}

\section{Numerical results and proof of consistency}
\label{sct:numerics}

\subsection{Implementation}

\subsubsection{Choice of the initial condition}
The numerical scheme is phrased in Lagrangian coordinates:
the discretization $\theh=(\xi_0,\xi_1,\ldots,\xi_K)$ of the reference domain $[0,M]$ is fixed, 
whereas the corresponding grid points $\xvec^n=(x^n_1,\ldots,x^n_{K-1})\in\xseq$ on the interval $I$ evolve in (discrete) time. 
In the numerical experiments that follows,
our choice for the discretization of the initial condition is to use an equidistant grid $\xvec^0$ with $K$ vertices on $I$,
\begin{align*}
 x^0_k = a+k(b-a)/K,
\end{align*}
and an accordingly adapted mesh $\theh$ on $[0,M]$, with
\begin{align*}
 \xi_k = U^0(x^0_k),\quad \text{where} \quad
 U^0(x)=\int_a^xu^0(y) \dd y \quad\text{for all $x\in I$}
\end{align*}
is the initial datum's distribution function.
This discretization has the property that
\begin{align*}
 \int_{x^0_{k-1}}^{x^0_k} u^0(x)\dd x = \int_{x^0_{k-1}}^{x^0_k} u_\Delta^0(x)\dd x
 \quad \text{for all $k=1,\ldots,K$}.
\end{align*}

\subsubsection{Time stepping}
Each (time) step in the numerical scheme consists of solving the system \eqref{eq:el} of Euler-Lagrange equations.
In practice, this is done with a damped Newton method, 
which guarantees that the constraint $\xvec_\Delta^n\in\xseq$ --- i.e., that $a<x^n_1<\cdots<x^n_{K-1}<b$ ---
is propagated from the $n-1$st to the $n$th iterate.
To be more precise, recall that
\begin{align*}
 \grd\nrj_{\theh,\tau}(\xvec) = \frac1\tau \Wmat(\xvec-\xvec_\Delta^{n-1}) + \grd\nrj_\theh(\xvec)
\end{align*}
is the functional whose unique root in $\xseq$ defines the $n$th time iterate $\xvec_\Delta^n$,
and that
\begin{align*}
 \hess\nrj_{\theh,\tau}(\xvec) = \frac1\tau\Wmat + \hess\nrj_\theh(\xvec)
\end{align*}
is its Jacobian.
Given $\xvec_\Delta^{n-1}$, we calculate $\xvec_\Delta^n$ by means of the following algorithm:
\begin{align*}
 & \xvec :=\xvec_\Delta^{n-1}; \\
 & \mathsf{repeat} \\
 & \qquad \dd\xvec := - \big(\hess\nrj_{\theh,\tau}(\xvec)\big)^{-1}\grd\nrj_{\theh,\tau}(\xvec);\\
 & \qquad \mathsf{while}\;\xvec + \dd\xvec \notin \xseq 
 \qquad \dd\xvec: = 0.5\dd\xvec; 
 \qquad \mathsf{end}; \\
 & \qquad \xvec := \xvec + \dd\xvec; \\
 & \mathsf{until}\;
 \|\dd\xvec\|_{l^1} < \mathsf{tol} \quad\text{and}\quad \|\grd\nrj_{\theh,\tau}(\xvec)\|_{l^1} < \mathsf{tol}; \\
 & \xvec_{\Delta}^{n} := \xvec.
\end{align*}
In our experiments, we use $\mathsf{tol} = 10^{-8}$. 
For the evaluation of $\grd\nrj_{\theh,\tau}(\xvec)$ above, an explicit expression for the integrals
\begin{align*}
 \partial_{x_m} \bigg(\int_0^M V\circ\convX_\theh[\xvec]\dd\xi\bigg) = \int_0^M V_x\circ\convX_\theh[\xvec]\hatf_m(\xi)\dd \xi
\end{align*}
is needed, see \eqref{eq:gradient}.
Denoting by $\mathfrak{V}$ an anti-derivative of $V$,
one finds
\begin{align*}
 \int_{\xi_{m-1}}^{\xi_m} V_x\circ\convX_\theh[\xvec]\hatf_m(\xi)\dd \xi
 &= \frac{\delta_m}{x_m-x_{m-1}}\int_{x_{m-1}}^{x_m} V_x(x)\frac{x-x_{m-1}}{x_m-x_{m-1}}\dd x \\
 &= \frac{\delta_m}{x_m-x_{m-1}} \bigg(V(x_m) - \frac{\mathfrak{V}(x_m)-\mathfrak{V}(x_{m-1})}{x_m-x_{m-1}}\bigg),
\end{align*}
and analogously for the integral from $\xi_m$ to $\xi_{m+1}$.
In combination, we obtain
\begin{align*}
 \partial_{x_k}\int_0^M V\circ\convX_\theh[\xvec](\xi)\dd\xi
 &=\frac{-\delta_k}{(x_k-x_{k-1})^2}\big(\mathfrak{V}(x_k)-\mathfrak{V}(x_{k-1})\big)
 +\frac{\delta_{k-1}}{(x_{k+1}-x_k)^2}\big(\mathfrak{V}(x_{k+1})-\mathfrak{V}(x_k)\big) \\
 &\quad+V(x_k)\left(\frac{\delta_k}{x_k-x_{k-1}}-\frac{\delta_{k+1}}{x_{k+1}-x_k}\right).
\end{align*}
A similar expression is available for the respective contribution to the Hessian $\hess\nrj_{\theh,\tau}$:
\begin{align*}
 \begin{split}
   &\partial_{x_m}\partial_{x_k}\bigg(\int_0^M V\circ\convX_\theh[\xvec]\dd\xi\bigg) \\
   &=
   \begin{cases}
     \frac{2\delta_k}{(x_k-x_{k-1})^3}\big(\mathfrak{V}(x_k)-\mathfrak{V}(x_{k-1})\big)
     +\frac{2\delta_{k+1}}{(x_{k+1}-x_{k})^3}\big(\mathfrak{V}(x_{k+1})-\mathfrak{V}(x_{k})\big)\\
     -2V(x_k)\left(\frac{\delta_k}{(x_k-x_{k-1})^2}+\frac{\delta_{k+1}}{(x_{k+1}-x_{k})^2}\right)
     -V_x\left(\frac{\delta_k}{(x_k-x_{k-1})^2}-\frac{\delta_{+1}k}{(x_{k+1}-x_{k})^2}\right),
     & m=k\\
     \smallskip
     -\frac{2\delta_k}{(x_k-x_{k-1})^3}\big(\mathfrak{V}(x_k)-\mathfrak{V}(x_{k-1})\big)
     +\frac{\delta_k}{(x_k-x_{k-1})^2}\left(V(x_k)+V(x_{k-1})\right),
     & m = k-1 \\
     \smallskip
     0, & \text{otherwise}.
   \end{cases}
 \end{split}
\end{align*}

\subsection{Numerical experiments and convergence}
The following numerical experiments are performed for the porous medium equation with quadratic nonlinearity,
\begin{align*}
 \partial_tu = (u^2)_{xx} + (V_xu)_x,
\end{align*}
on the interval $I=[-1,1]$.
For the potential $V$, we choose
\begin{align*}
 V(x)=-\frac1\pi\cos(\pi x),
\end{align*}
and as initial datum, we take the following function of unit mass $M=1$:
\begin{align}
 \label{eq:id}
 u^0(x) = C\big(-\cos(2\pi x)+1.5\big)\big((x+0.5)^4+1\big)\quad \text{with}\quad C=\frac{240-280\pi^2+423\pi^4}{80\pi^4}.
\end{align}

\subsubsection{Reference Solution}
\label{exp0}
Our numerical reference resolution is calculated with $K=5000$ spatial grid points and a time step size $\tau=10^{-2}$.
Figure \ref{fig:fig1}/left shows snapshots of the reference solution's spatial density after the first couple of time steps.
One observes the typical behaviour for nonlinear drift diffusion equations: 
on a very short time scale, diffusion reduces the extrema of the initial mass distribution;
subsequently, the drift dominates and transports the mass towards the equilibrium (dotted line) on a longer time scale.
Figure \ref{fig:fig1}/right displays the corresponding particle trajectories in the Lagrangian picture, 
i.e., how the points $x_k^n$ move with (discrete) time $n$ for fixed $k$.
\begin{figure}
 \centering
 \subfigure{\includegraphics[width=0.48\textwidth]{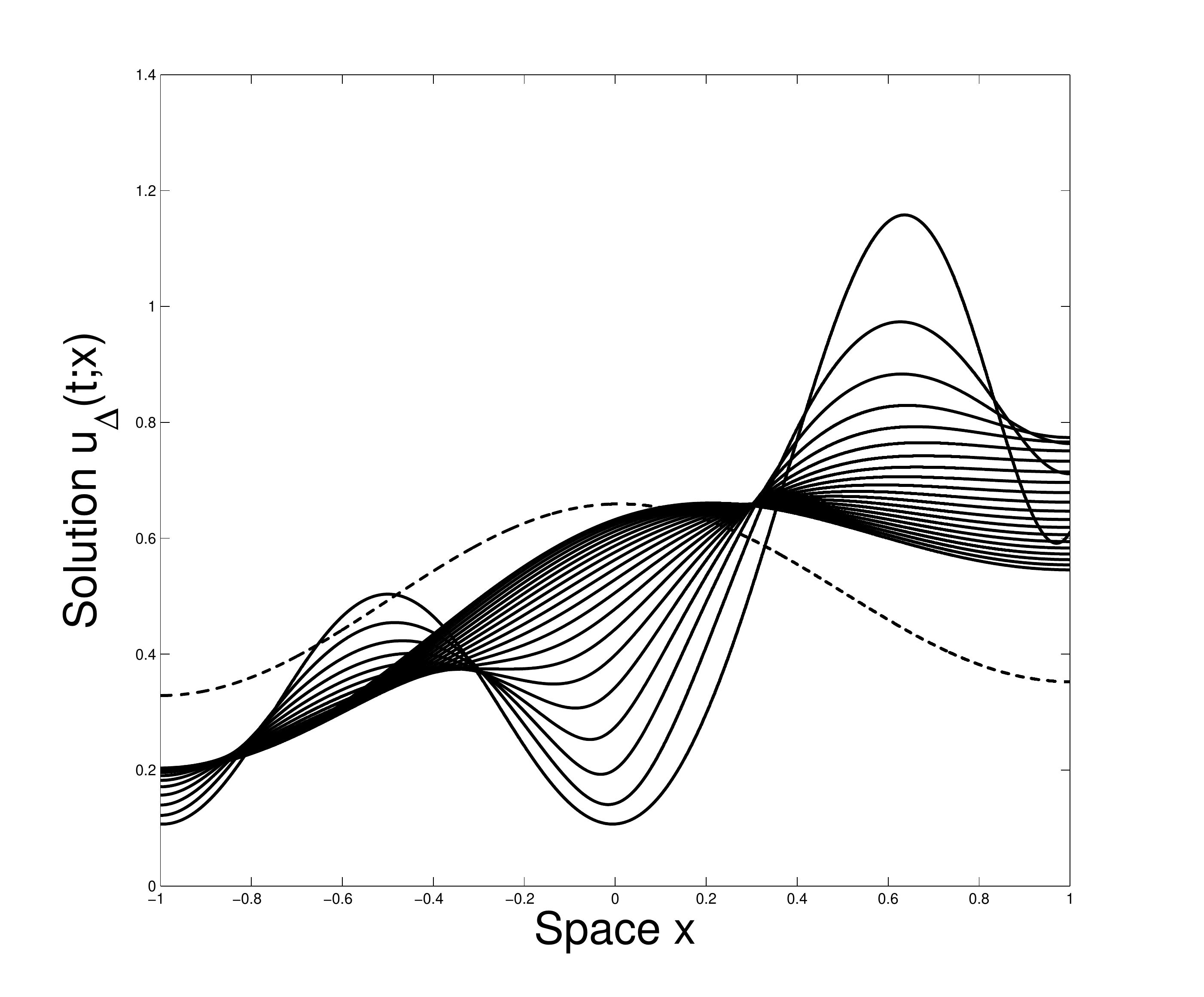}}
 \subfigure{\includegraphics[width=0.42\textwidth]{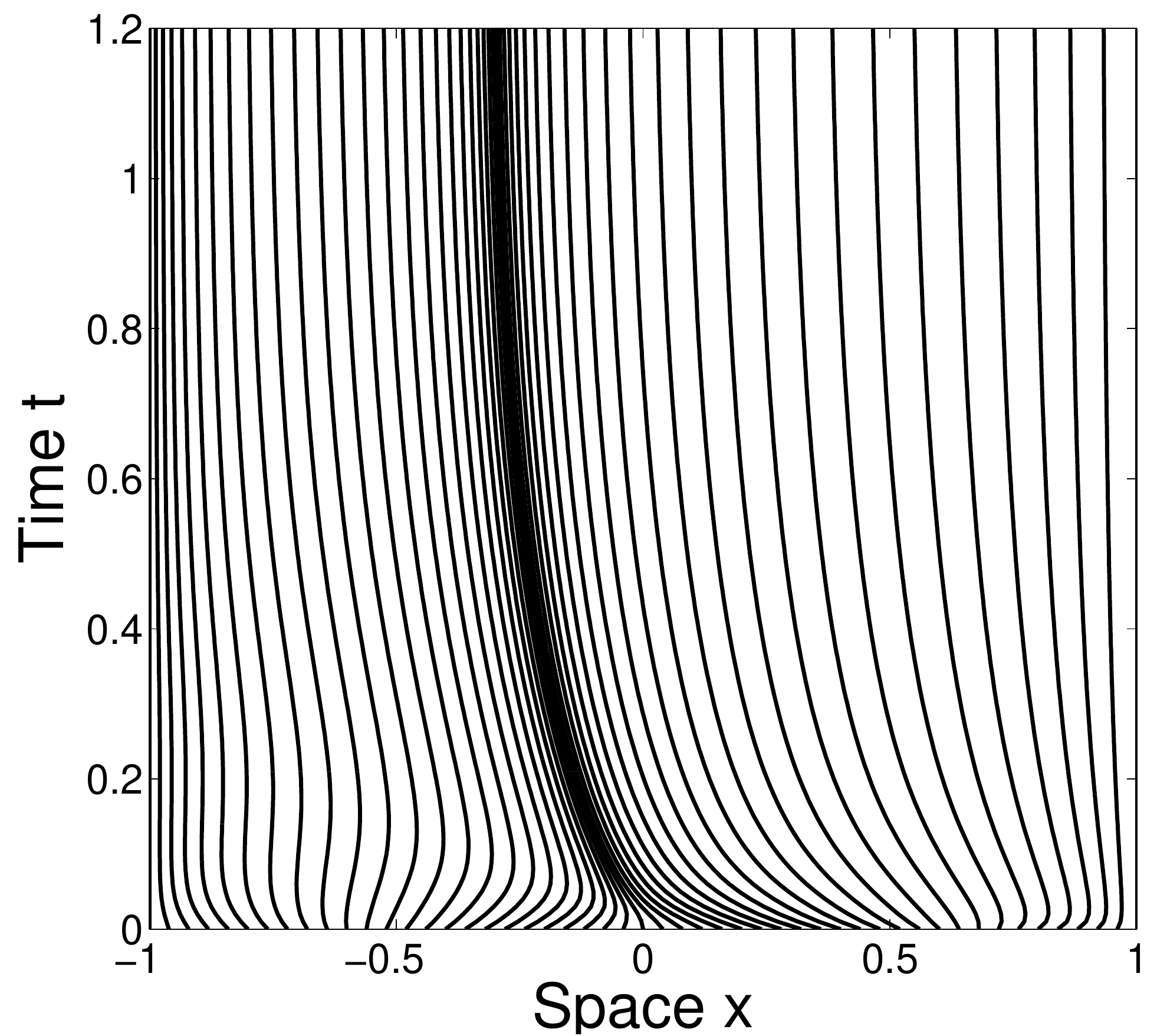}}
 \caption{Left: evolution of the (reference) solution $u_\Delta$ with initial condition \eqref{eq:id} at times $t=0,\tau,\ldots,20\tau=0.2$,
   with time step $\tau=10^{-2}$ and $K=5000$ grid points.
   The dotted line shows the stationary solution.
   Right: associated particle trajectories.}
 \label{fig:fig1}
\end{figure}

\subsubsection{Fixed $\tau$}
\label{exp1}
In a first series of experiments, we fix the time step $\tau=10^{-2}$ and vary the number of spatial grid points $K$.
In Figure \ref{fig:fig2}/left, the corresponding $L^1$-distances to the reference solution $u_{\text{ref}}$ obtained in \ref{exp0} above are shown as a function of time.
Note that the counter-intuitive dramatic decay of the error for small times is explained by the strong contractivity of the nonlinear diffusion in $L^1$
in a neighborhood of the initial condition.
Figure \ref{fig:fig2}/right shows the $L^1$-errors at $T=0.2$.
The observed convergence rate is of order $K^{-1}$.

\begin{figure}
 \centering
 \subfigure{\includegraphics[width=0.4\textwidth]{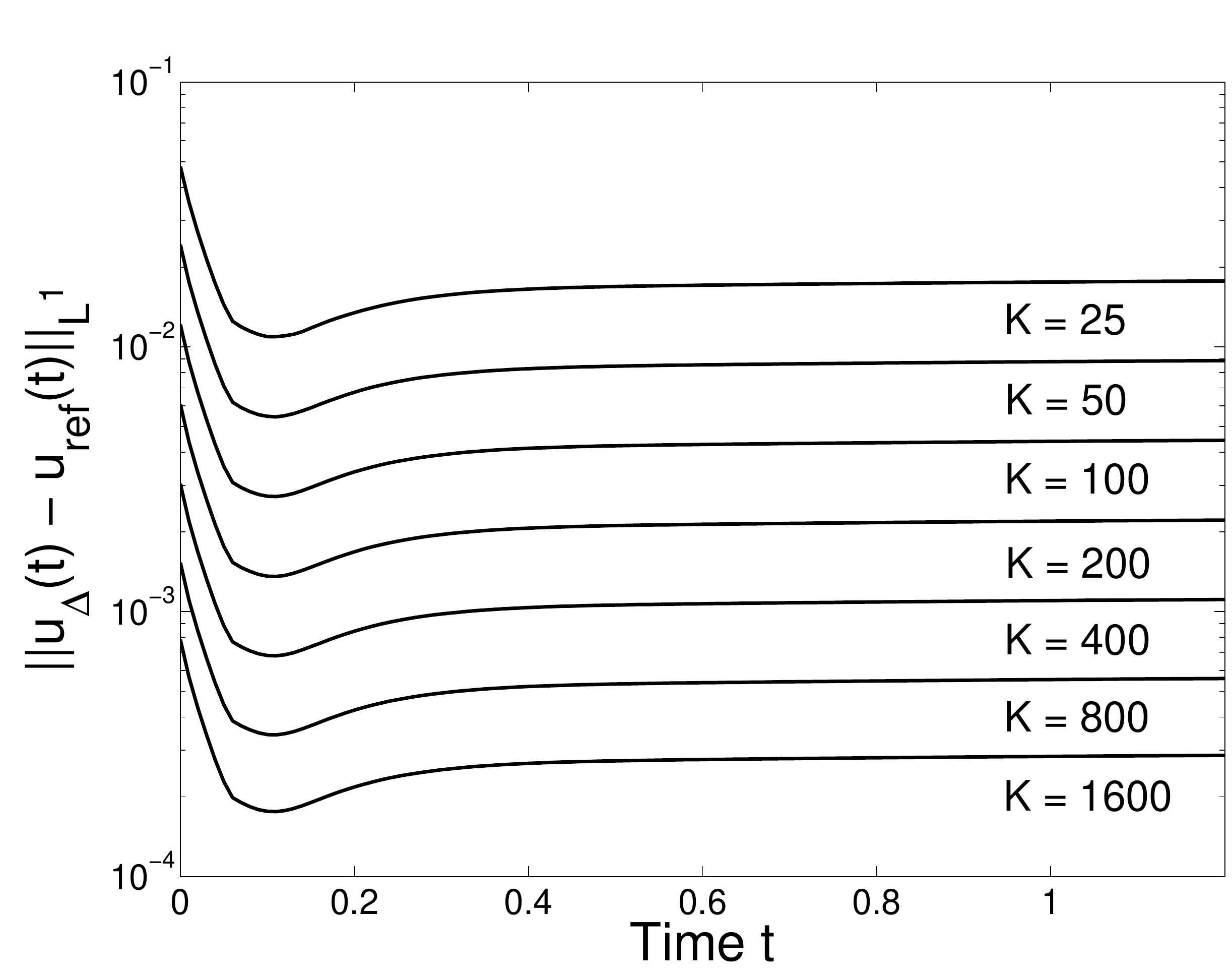}}
 \subfigure{\includegraphics[width=0.5\textwidth]{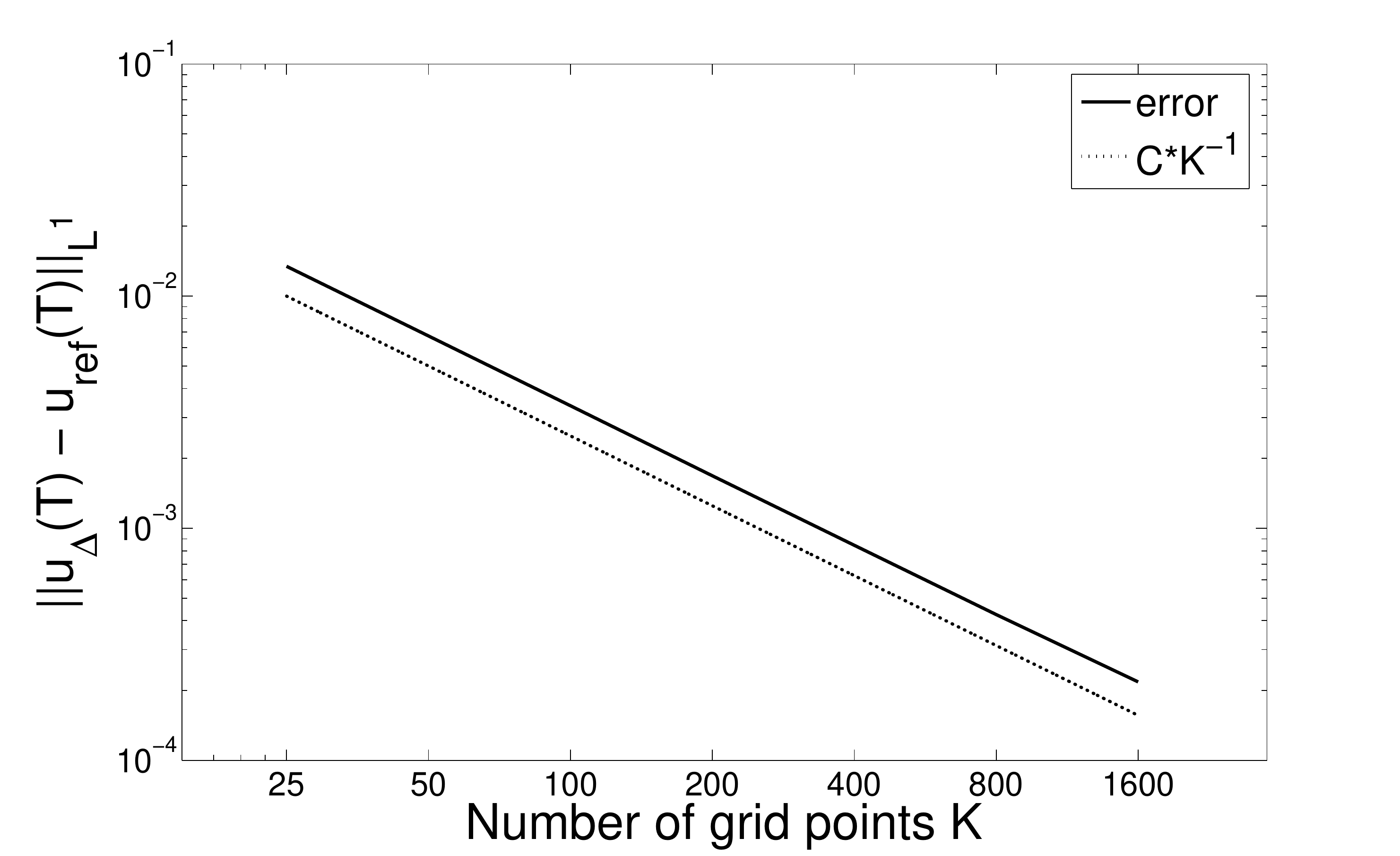}}
 \caption{Numerical error analysis with fixed time step $\tau=10^{-2}$, using $K=25,50,100,200,400,800,1600$ grid points. 
   Left: evolution of the $L^1$-error $\|\bar u_{\Delta}(t)-u_{\text{ref}}(t)\|_{L^1(I)}$. 
   Right: order of convergence at terminal time $T=0.2$.}
 \label{fig:fig2}
\end{figure}

\subsubsection{Fixed parabolic mesh ratio}
\label{exp3}
Next, we study the decay of the $L^1$-error under mutual refinement of space and time.
As it is standard in numerical experiments on parabolic equations, we fix the parabolic mesh ratio $K^2\tau$.
The value of this ratio is chosen such that the inverse CFL condition is satisfied in every experiment.
In Figure \ref{fig:fig4}, the error is plotted  --- similar as in the previous experiment ---
as a function of time (left) and at the fixed terminal time $T=0.2$ (right),
both for various choices of $\tau$.
The observed order of convergence is $\sqrt{\tau}$, which is in agreement with the result of experiment \ref{exp1}.
\begin{figure}
 \centering
 \subfigure{\includegraphics[width=0.45\textwidth]{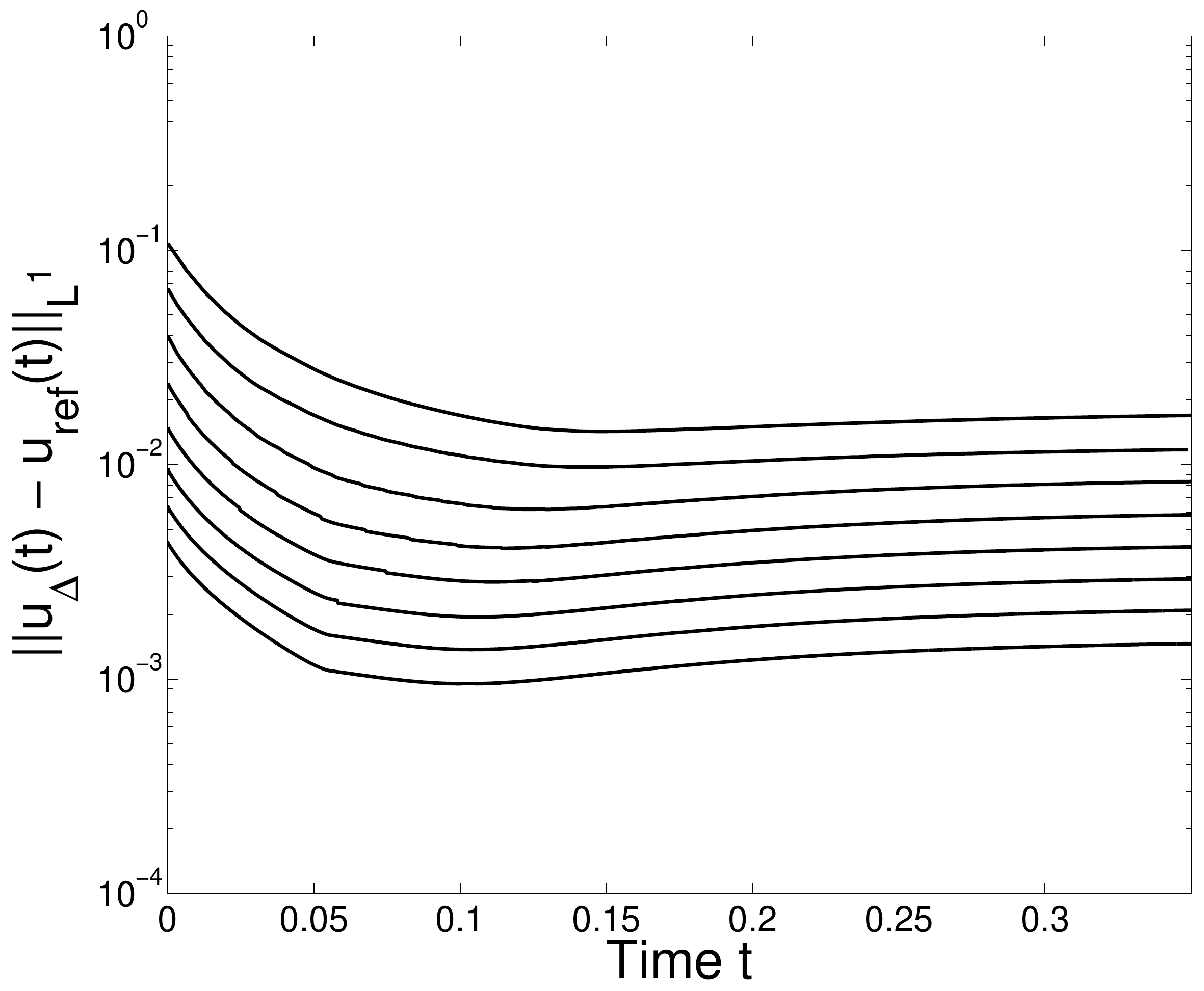}}
 \subfigure{\includegraphics[width=0.45\textwidth]{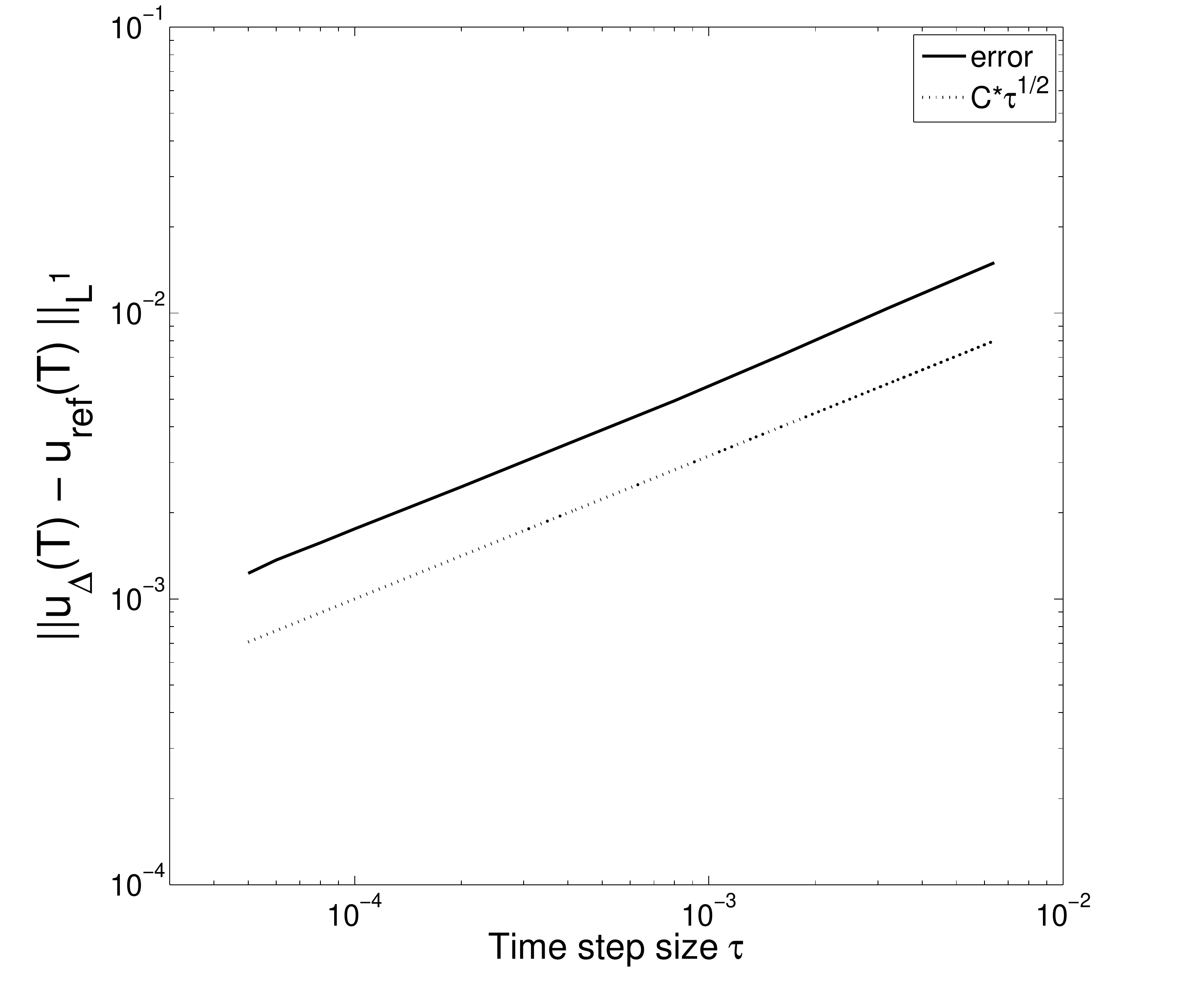}}
 \caption{Numerical error analysis with fixed parabolic mesh ratio $K^2\tau\approx 0.257$, 
   using $\tau=5\cdot10^{-5},10^{-4},5\cdot10^{-4},10^{-3},5\cdot10^{-3},10^{-2},5\cdot10^{-2},10^{-1}$. 
   Left: evolution of the $L^1$-error $\|\bar u_{\Delta}(t)-u_{\text{ref}}(t)\|_{L^1(I)}$. 
   Right: order of convergence at terminal time $T=0.2$.}
 \label{fig:fig4}
\end{figure}

\subsubsection{Weakly convergent initial datum}
\label{expnew}
In order to illustrate that it sufficies to approximate the original initial condition $u^0$ by its discretizations $u_\Delta^0$ just \emph{weakly} in $L^1(I)$,
we use perturbed discrete initial data $u_{\Delta,\eps}^0$ that are biased by high-frequency oscillations of fixed amplitude $0.1$, 
as indicated in Figure \ref{fig:fignew}/left.
As expected, the perturbation becomes almost invisible already after the first time step,
and the discrete solution $u_{\Delta,\eps}$ is indistinguishable from the one computed with unperturbed initial conditions $u_\Delta^0$.
\begin{figure}
 \centering
 \subfigure{\includegraphics[width=0.45\textwidth]{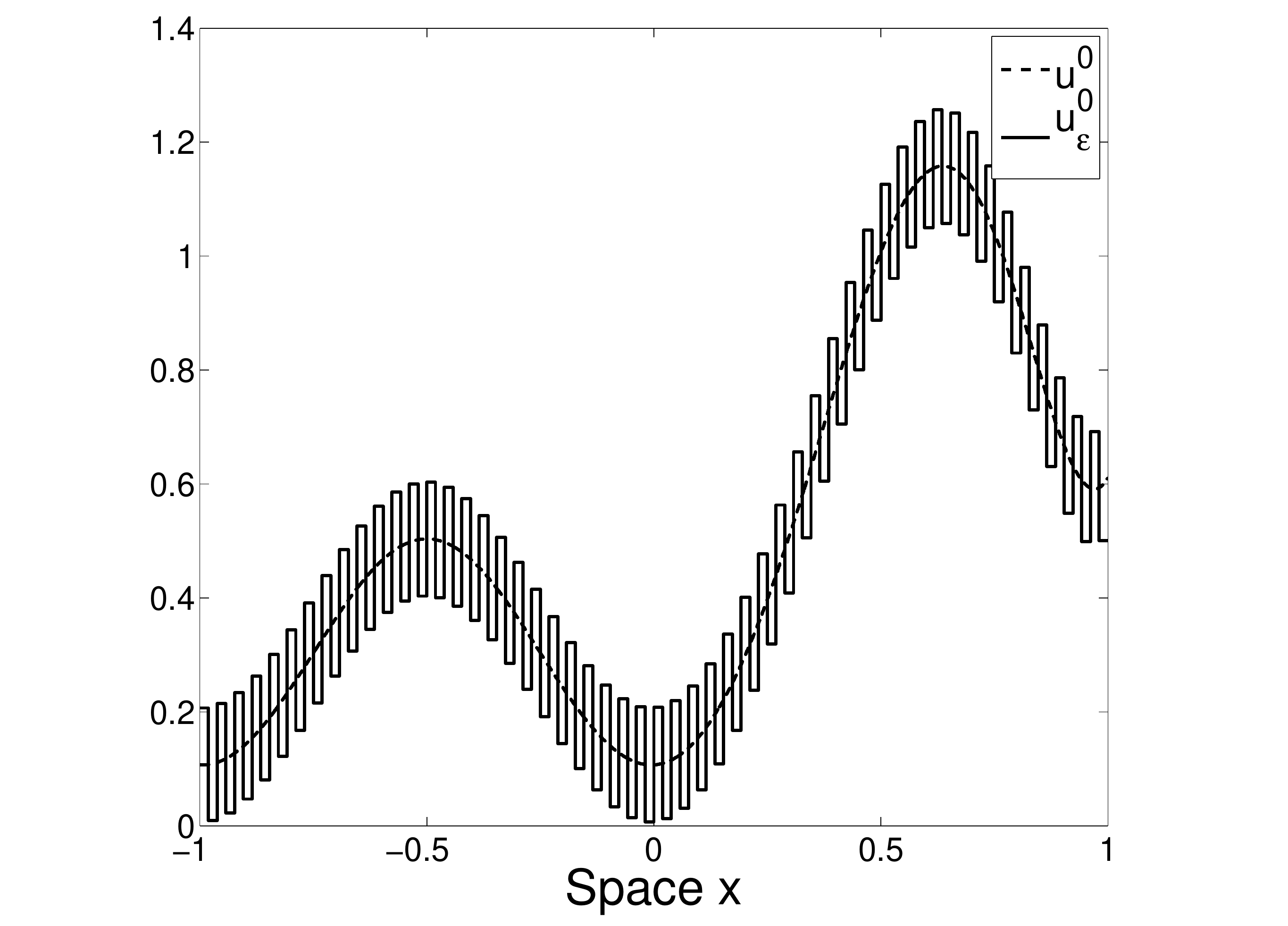}}
 \subfigure{\includegraphics[width=0.45\textwidth]{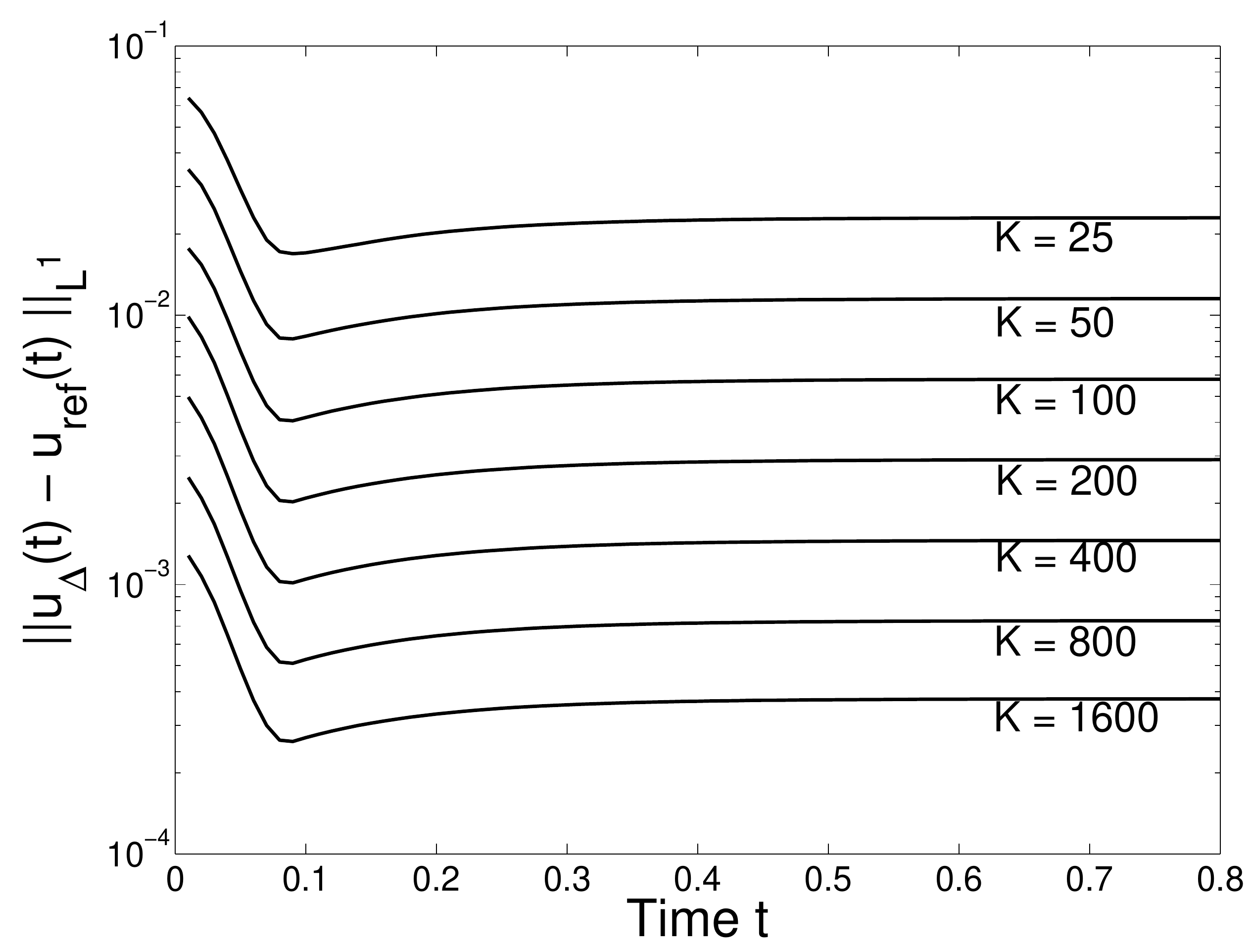}}
 \caption{Left: initial condition $u_{\Delta,\eps}^0$ with high frequency perturbation. 
   Right: numerical error analysis for discrete solutions with the discontinuous initial datum from \eqref{eq:discic},
   using a fixed time step $\tau=10^{-2}$ and varying $K=25,50,100,200,400,800,1600$.}
 \label{fig:fignew}
\end{figure}

\subsubsection{A discontinuous initial datum}
\label{exp4}
For the last two series of experiments, we change the initial condition $u^0$.
This first series is carried out with the discontinuous inital datum
\begin{align}
 \label{eq:discic}
 u^0(x) = 
 \begin{cases} 
   0.1, &\text{if $|x|>0.75$ or $|x|<0.25$}, \\
   0.9, &\text{otherwise}.
 \end{cases}
\end{align}
Similar to experiment \ref{exp1}, we fix $\tau=10^{-2}$ and vary the number of grid points $K$.
Figure \ref{fig:fignew}/right displays the corresponding $L^1$-error over the time interval $t\in[0,0.8]$.
In contrast to experiment \ref{exp1}, the approximation error is zero initially,
since the step function $u^0$ can be discretized exactly.
However, the error jumps to a positive value (that is of the same order as the initial error in experiment \ref{exp1}) in the first time step.
Afterwards, the qualitative behaviour is very similar to that in experiment \ref{exp1}.
The observed order of convergence (at $T=0.2$) is again $K^{-1}$.

\subsubsection{A non-positive initial datum}
\label{exp5}
For this last series of experiment, we consider the initial condition
\begin{align}
 \label{eq:compic}
 u^0(x) = \big(-\cos(2\pi x)+1.5\big)\big((x+0.5)^4+1\big)\times
 \begin{cases}-(x-0.5)(x+0.5) & |x|\leq 0.5\\0 & |x|> 0.5\end{cases},
\end{align}
which vanishes outside of the subinterval $[-0.5,0.5]\subset I$.
The numerical scheme is not directly applicable to $u^0$, 
but to any of its strictly positive approximations $u^0+\varepsilon$, see Figure \ref{fig:fig5}/left.
The qualitative numerical results at $T=0.6$ for various choices of $\varepsilon>0$ are given in Figure \ref{fig:fig5}/right.
\begin{figure}
 \centering
 \subfigure{\includegraphics[width=0.47\textwidth]{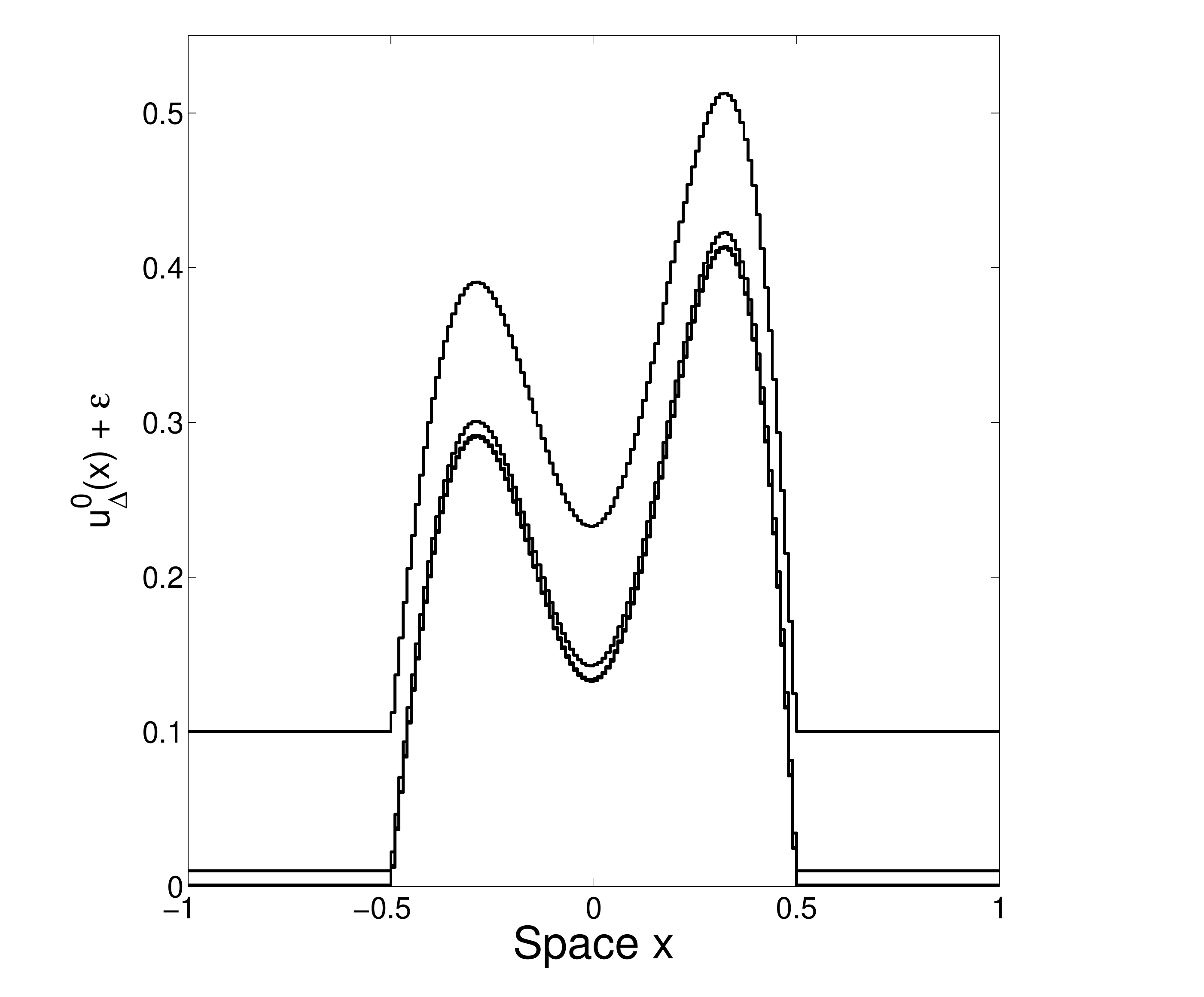}}
 \hspace{-0.05\textwidth}
 \subfigure{\includegraphics[width=0.55\textwidth]{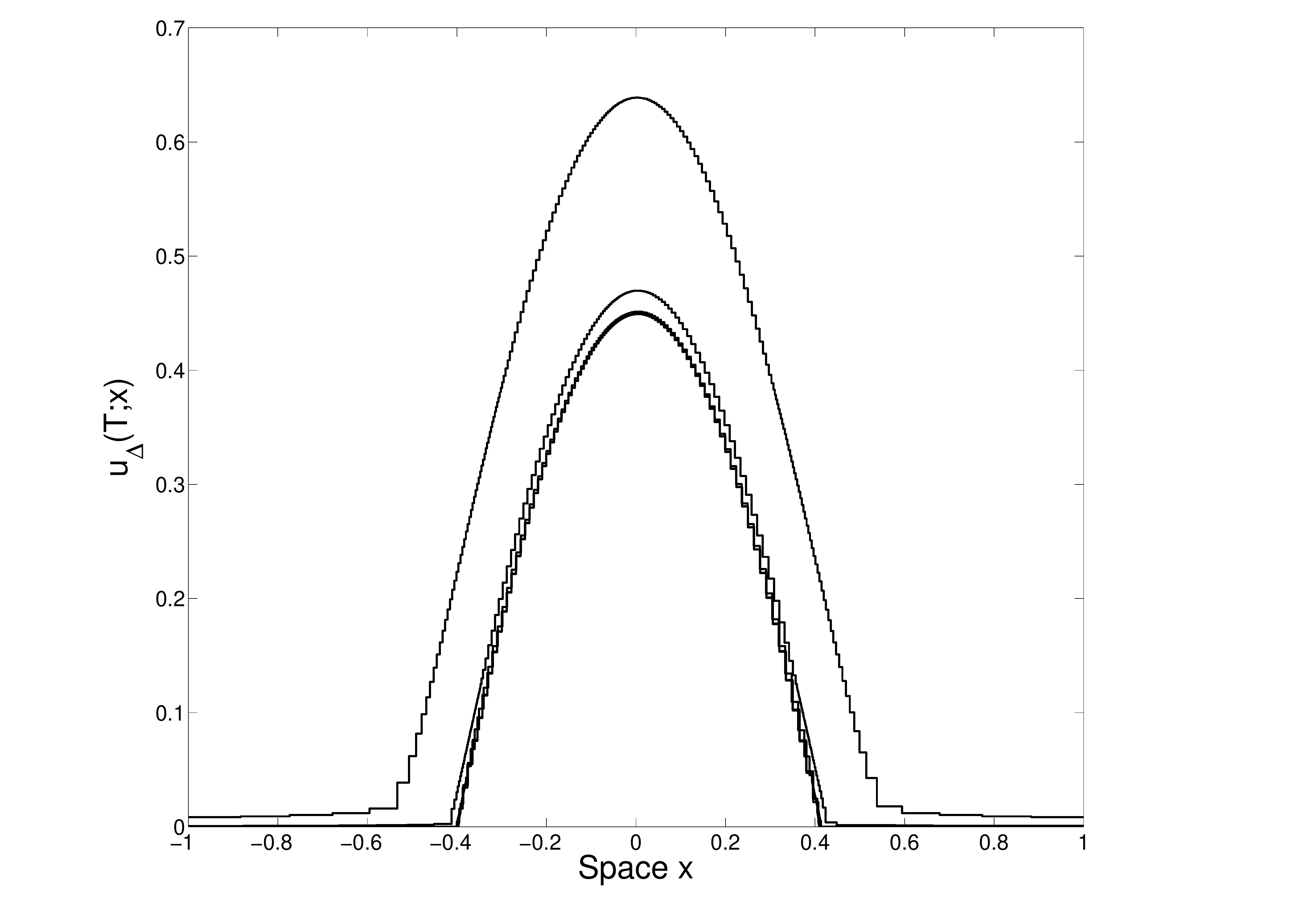}}
 \caption{The merely non-negative initial condition $u^0$ from \eqref{eq:compic} is approximated by strictly positive data $u^0+\eps$.
   Left: discrete initial profiles for $\eps=10^{-1},10^{-2},10^{-3},10^{-4},10^{-5}$.
   Right: qualitative behaviour of corresponding discrete solutions at $T=0.6$,
   using $\tau=10^{-3}$, $K=200$.}
 \label{fig:fig5}
\end{figure}

\subsection{Order of consistency}
The experimental observations that the discrete solutions seem to approximate the reference solution
with an error of order $K\tau$ can be supported theoretically by the following consistency consideration.

Assume that $X:[0,T]\times[0,M]\to I$ is a smooth solution of \eqref{eq:dde}.
Consider a discretization $\Delta=(\tau;\theh)$ that is equidistant w.r.t.\ $\xi$, 
i.e., $\xi_k=Mk/K$ with some $K\in\setN$ for all $k=0,1,2,\ldots,K$,
and consequently $\delta:=\delta_1=\cdots=\delta_K$.

From $\theX$, we define a discrete ``pseudo-solution'' $\xvec_\Delta$ by restriction, i.e., $x^n_k:=\theX(n\tau,k\delta)$.
We are going to show that $\xvec_\Delta$ satisfies the discrete evolution equation \eqref{eq:el} up to an error of order $\delta(\bigo(\tau)+\bigo(\delta^2))$.
To this end, we perform a Taylor expansion of $\theX$ around a fixed point $\apt:=(n\tau,k\delta)$ w.r.t.\ $\xi$:
\begin{align*}
 x^n_{k\pm1} &= \theX(\apt) \pm\delta \theX_\xi(\apt) + \frac{\delta^2}2\theX_{\xi\xi}(\apt) \pm \frac{\delta^3}6\theX_{\xi\xi\xi}(\apt) + \bigo(\delta^4).
\end{align*}
For the diffusion term, we find
\begin{align*}
 \psi'\Big(\frac{x^n_{k\pm1}-x^n_k}\delta\Big)
 &= \psi'(\theX_\xi(\apt)) \pm \frac\delta2\psi''(\theX_\xi(\apt))\theX_{\xi\xi}(\apt) \\
 & \qquad + \frac{\delta^2}2\Big(\frac13\psi''(\theX(\apt))\theX_{\xi\xi\xi}(\apt)+\frac14\psi'''(\theX(\apt))\theX_{\xi\xi}(\apt)^2\Big) + \bigo(\delta^3),
\end{align*}
so that
\begin{align*}
 \psi'\Big(\frac{x^n_{k+1}-x^n_k}\delta\Big) -\psi'\Big(\frac{x^n_k-x^n_{k-1}}\delta\Big)
 &= \delta\psi''(\theX_\xi(\apt))\theX_{\xi\xi}(\apt) + \bigo(\delta^3).
\end{align*}
For the drift term, we obtain
\begin{align*}
 \int_{k\delta}^{(k\pm1)\delta}V_x\big(\convX_\theh[\xvec_\Delta](\xi)\big)\hatf_k(\xi)\dd\xi
 & = \pm\delta \int_0^1 V_x\big((1-s)x^n_k + sx^n_{k\pm1}\big)(1-s)\dd s \\
 & = \pm\delta \int_0^1 \big[V_x(\theX(\apt)) \pm s\delta V_{xx}(\theX(\apt))\theX_\xi(\apt) + \bigo(\delta^2)\big](1-s)\dd s \\
 & = \pm\frac\delta2 V_x(\theX(\apt)) + \frac{\delta^2}6V_{xx}(\theX(\apt))\theX_\xi(\apt) + \bigo(\delta^3),
\end{align*}
so that
\begin{align*}
 \int_{(k-1)\delta}^{(k+1)\delta} V_x\big(\convX_\theh[\xvec_\Delta](\xi)\big)\hatf_k(\xi)\dd\xi
 &= \delta V_x(\theX(\apt)) + \bigo(\delta^3).
\end{align*}
In combination,
\begin{align}
 \label{eq:rightside}
 \big[\grd\nrj_\theh(\xvec)\big]_k = \delta\big(\psi'(\theX_\xi(\apt))_\xi + V_x(\theX(\apt)) + \bigo(\delta^2)\big).
\end{align}
Moreover, we have
\begin{align*}
 [\Wmat\xvec_\Delta^n]_k = \frac\delta6 x^n_{k-1} + \frac{2\delta}3x^n_k + \frac\delta6x^n_{k+1} = \delta \theX(\apt) + \bigo(\delta^3),
\end{align*}
and likewise
\begin{align*}
 [\Wmat\xvec_\Delta^{n-1}]_k = \delta \theX(\apt') + \bigo(\delta^3), 
\end{align*}
where $\apt'=((n-1)\tau,k\delta)$.
Finally, using
\begin{align*}
 \theX(\apt)-\theX(\apt') = \tau \theX_t(\apt) + \bigo(\tau^2),
\end{align*}
we obtain the relation
\begin{align}
 \label{eq:leftside}
 \frac1\tau\big[\Wmat(\xvec_\Delta^n-\xvec_\Delta^{n-1})\big]_k = \delta \big( \theX_t(\apt) + \bigo(\delta^2) + \bigo(\tau)\big).
\end{align}
Using the continuous evolution equation \eqref{eq:dde} in \eqref{eq:leftside} and \eqref{eq:rightside} leads to
\begin{align*}
 \frac1\tau\big[\Wmat(\xvec_\Delta^n-\xvec_\Delta^{n-1})\big]_k
 = - \big[\grd\nrj_\theh(\xvec)\big]_k + \delta\big(\bigo(\tau) + \bigo(\delta^2)\big).
\end{align*}
for all admissible $k$ and $n$.

\bibliographystyle{plain}
\bibliography{Horst}

\end{document}